\documentclass[12pt]{amsart}

\usepackage{hyperref}
\usepackage{amssymb,tikz}

\def\CC{\mathbb C}
\def\NN{\mathbb N}

\def\RR{\mathbb R}
\def\TT{\mathbb T}
\def\ZZ{\mathbb Z}

\def\Cc{\mathcal C}
\def\Dd{\mathcal D}
\def\Ee{\mathcal E}
\def\Ff{\mathcal F}
\def\Gg{\mathcal G}
\def\Hh{\mathcal H}
\def\Kk{\mathcal K}

\def\Oo{\mathcal O}
\def\Pp{\mathcal P}
\def\Tt{\mathcal T}

\def\ker{\operatorname{ker}}

\def\Ext{\operatorname{Ext}}

\def\id{\operatorname{id}}

\def\Aut{\operatorname{Aut}}
\def\lsp{\operatorname{span}}
\def\clsp{\operatorname{\overline{span\!}\,\,}}
\def\mod{\operatorname{mod}}

\def\MCE{\operatorname{MCE}}
\def\IC{\operatorname{IC}}
\newcommand{\Me}{\sim_{\mathrm{Me}}}
\newcommand{\FE}{\operatorname{FE}}
\newcommand{\sources}{\mathrm{src}}
\newcommand{\Lmin}{\Lambda^{\mathrm{min}}}
\newcommand{\im}{i}

\newcommand{\PAF}{\Pp_{\mathrm{I}}}
\newcommand{\PNAF}{\Pp_{\mathrm{I\hskip-0.5pt I}}}
\newcommand{\LAF}{\Lambda_{\mathrm{I}}}
\newcommand{\LNAF}{\Lambda_{\mathrm{I\hskip-0.5pt I}}}

\theoremstyle{plain}
\newtheorem{theorem}{Theorem}[section]
\newtheorem*{theorem*}{Theorem}
\newtheorem*{prop*}{Proposition}
\newtheorem{cor}[theorem]{Corollary}
\newtheorem{lemma}[theorem]{Lemma}
\newtheorem{prop}[theorem]{Proposition}

\theoremstyle{remark}
\newtheorem{rmk}[theorem]{Remark}

\newtheorem{example}[theorem]{Example}

\theoremstyle{definition}
\newtheorem{dfn}[theorem]{Definition}

\numberwithin{equation}{section}

\title[AF $k$-graph $C^*$-algebras]{When is the Cuntz-Krieger algebra of a higher-rank
graph approximately finite-dimensional?}

\author{D. Gwion Evans}

\address{Institute of Mathematics and Physics\\ Aberystwyth University\\ Penglais Campus \\
Aberystwyth\\ Ceredigion \\ SY23 3BZ \\ Wales \\ UK.}

\email{dfe@aber.ac.uk}

\author{Aidan Sims}

\address{School of Mathematics and Applied Statistics\\
University of Wollongong \\
NSW 2522\\
Australia}

\email{asims@uow.edu.au}

\date{\today}
\thanks{This research was supported by the Australian Research Council
and by an LMS travelling lecturer grant.}

\begin{document}

\begin{abstract}
We investigate the question: when is a higher-rank graph $C^*$-algebra approximately finite
dimensional? We prove that the absence of an appropriate higher-rank analogue of a cycle is
necessary. We show that it is not in general sufficient, but that it is sufficient for higher-rank
graphs with finitely many vertices. We give a detailed description of the structure of the
$C^*$-algebra of a row-finite locally convex higher-rank graph with finitely many vertices. Our
results are also sufficient to establish that if the $C^*$-algebra of a higher-rank graph is AF,
then its every ideal must be gauge-invariant. We prove that for a higher-rank graph $C^*$-algebra
to be AF it is necessary and sufficient for all the corners determined by vertex projections to be
AF. We close with a number of examples which illustrate why our question is so much more difficult
for higher-rank graphs than for ordinary graphs.
\end{abstract}

\maketitle

\section{Introduction}

A directed graph $E$ consists of countable sets $E^0$ and $E^1$ and maps $r,s : E^1 \to E^0$. We
call elements of $E^0$ \emph{vertices} and elements of $E^1$ \emph{edges} and think of each $e \in
E^1$ as an arrow pointing from $s(e)$ to $r(e)$. When $r^{-1}(v)$ is finite and nonempty for all
$v$, the graph $C^*$-algebra $C^*(E)$ is the universal $C^*$-algebra generated by a family of
mutually orthogonal projections $\{p_v : v \in E^0\}$ and a family of partial isometries $\{s_e : e
\in E^1\}$ such that $s^*_e s_e = p_{s(e)}$ for all $e \in E^1$ and $p_v = \sum_{r(e) = v} s_e
s^*_e$ for all $v \in E^0$ \cite{EW1980, KPRR1997}.

Despite the elementary nature of these relations, the class of graph $C^*$-algebras is quite rich.
It includes, up to strong Morita equivalence, all AF algebras \cite{Drinen2000, Tyler2004}, all
Kirchberg algebras whose $K_1$ group is free abelian \cite{Szymanski2002} and many other
interesting $C^*$-algebras besides \cite{HS2002, HS2003a}. We know this because we can read off a
surprising amount of the structure of a graph $C^*$-algebra (for example its $K$-theory
\cite{PR1996, RS2004}, and its whole primitive ideal space \cite{HS2004}) directly from the graph.
In particular, a graph $C^*$-algebra is AF if and only if the graph contains no directed cycles
\cite[Theorem 2.4]{KPR1998}. Moreover, if $E$ contains a directed cycle and $C^*(E)$ is simple,
then $C^*(E)$ is purely infinite. So every simple graph $C^*$-algebra is classifiable either by
Elliott's theorem or by the Kirchberg-Phillips theorem.

In 2000, Kumjian and Pask introduced higher-rank graphs, or $k$-graphs, and their $C^*$-algebras
\cite{KP2000} as a generalisation of graph algebras designed to model Robertson and Steger's
higher-rank Cuntz-Krieger algebras \cite{RS1999a}. These have proved a very interesting source of
examples in recent years \cite{DavidsonYang:CJM09, PRRS2006}, but remain far less-well understood
than their 1-dimensional counterparts, largely because their structure theory is much more
complicated. In particular, a general structure result for simple $k$-graph algebras is still
lacking; even a satisfactory characterisation of simplicity itself is in full generality fairly
recent \cite{Shotwell:xx08}. The examples of \cite{PRRS2006} show that there are simple $k$-graph
algebras which are neither AF nor purely infinite, indicating that the question is more complicated
than for directed graphs. Some fairly restrictive sufficient conditions have been identified which
ensure that a simple $k$-graph $C^*$-algebra is AF \cite[Lemma 5.4]{KP2000} or is purely infinite
\cite[Proposition 8.8]{Sims2006a}, but there is a wide gap between the two.

Deciding whether a given $C^*$-algebra is AF is an interesting and notoriously difficult problem.
The guiding principle seems to be that if, from the point of view of its invariants, it looks AF
and it smells AF, then it is probably AF. This point of view led to the discovery and analyses of
non-AF fixed point subalgebras of group actions on non-standard presentations of AF algebras
initiated by  \cite{Blackadar:AoM90} and \cite{Kumjian88} and continued by
\cite{EvansKishimoto:JFA91, BratteliElliottEtAl:KTh94} and others.  Numerous powerful AF
embeddability theorems (the canonical example is \cite{PimsnerVoiculescu:JOT80a}; and more recently
for example \cite{Katsura:JFA02, Brown:JFA98, Spielberg:JFA88}) have also been uncovered.  These
results demonstrate that algebraic obstructions --- beyond the obvious one of stable finiteness ---
to approximate finite dimensionality of $C^*$-algebras are hard to come by. On the other hand,
proving that a given $C^*$-algebra is AF can be a highly nontrivial task (cf.
\cite{BratteliElliottEtAl:KTh94, BratteliEvansEtAl:ETDS93} and the series of penetrating analyses
of actions of finite subgroups of $\mathrm{SL}_2(\ZZ)$ on the irrational rotation algebra initiated
by \cite{BEEK91, BratteliKishimoto:CMP92, Walters95} and culminating in
\cite{EchterhoffLuckEtAl:Crelle10}). Moreover, non-standard presentations of AF algebras have found
applications in classification theory \cite{PimsnerVoiculescu:JOT80a}, and also to long-standing
questions such as the Powers-Sakai conjecture \cite{Kishimoto:JFA03}.

In this paper, we consider more closely the question of when a $k$-graph $C^*$-algebra is AF. The
question is quite vexing, and we have not been able to give a complete answer (see
Example~\ref{eg:skew-pullback}). However, we have been able to weaken the existing necessary
condition for the presence of an infinite projection, and also to show that for a $k$-graph
$C^*$-algebra to be AF, it is necessary that the $k$-graph itself should contain no directed
cycles; indeed, we identify a notion of a higher-dimensional cycle the presence of which precludes
approximate finite dimensionality of the associated $C^*$-algebra. Our results are sufficiently
strong to completely characterise when a unital $k$-graph $C^*$-algebra is AF, and to completely
describe the structure of unital $k$-graph $C^*$-algebras associated to row-finite $k$-graphs. We
also provide some examples confirming some earlier conjectures of the first author. Specifically,
we construct a $2$-graph $\Lambda$ which contains no cycles and in which every infinite path is
aperiodic, but such that $C^*(\Lambda)$ is finite but not AF, and we construct an example of a
$2$-graph which does not satisfy \cite[Condition~(S)]{EvansPhD} but does satisfy
\cite[Condition~($\Gamma$)]{EvansPhD} and whose $C^*$-algebra is AF. We close with an intriguing
example of a $2$-graph $\LNAF$ whose infinite-path space contains a dense set of periodic points,
but whose $C^*$-algebra is simple, unital and AF-embeddable, and shares many invariants with the
$2^\infty$ UHF algebra. If, as seems likely, the $C^*$-algebra of $\LNAF$ is strongly Morita
equivalent to the $2^\infty$ UHF algebra, it will follow that the structure theory of simple
$k$-graph algebras is much more complex than for graph algebras.

We remark that a proof that $C^*(\LNAF)$ is indeed AF would provide another interesting
non-standard presentation of an AF algebra. It would open up the possibility that known
constructions for $k$-graph $C^*$-algebras might provide new insights into questions about AF
algebras.

\subsection*{Acknowledgements}
We thank David Evans for suggesting the title of the paper as a research question. We also thank
Bruce Blackadar, Alex Kumjian, Efren Ruiz and Mark Tomforde for helpful discussions, and Andrew
Toms and Wilhelm Winter for helpful email correspondence. Finally, Aidan thanks Gwion for his warm
hospitality in Rome and again in Aberystwyth.

\section{Background}

We introduce some background relating to $k$-graphs and their $C^*$-algebras. See \cite{KP2000,
RSY2003, RSY2004} for details.

\subsection{Higher-rank graphs}
Fix an integer $k > 0$. We regard $\NN^k$ as a semigroup under pointwise addition with identity
element denoted 0. When convenient, we also think of it as a category with one object. We denote
the generators of $\NN^k$ by $e_1, \dots e_k$, and for $n \in \NN^k$ and $i \le k$ we write $n_i$
for the $i$\textsuperscript{th} coordinate of $n$; so $n = (n_1, n_2, \dots, n_k) = \sum^k_{i=1}
n_i e_i$. For $m, n \in \NN^k$, we write $m \le n$ if $m_i \le n_i$ for all $i$, and we write
$m\vee n$ for the coordinatewise maximum of $m$ and $n$, and $m \wedge n$ for the coordinatewise
minimum of $m$ and $n$. Observe that $m \wedge n \le m,n \le m \vee n$, and that $m' := m - (m
\wedge n)$ and $n' := n - (m \wedge n)$ is the unique pair such that $m - n = m' - n'$ and $m'
\wedge n' = 0$. For $n \in \NN^k$, we write $|n|$ for the \emph{length} $|n| = \sum^k_{i=1} n_i$ of
$n$.

As introduced in \cite{KP2000}, a \emph{graph of rank $k$} or a $k$-graph is a countable small
category $\Lambda$ equipped with a functor $d : \Lambda \to \NN^k$, called the \emph{degree
functor}, which satisfies the \emph{factorisation property}: for all $m,n \in \NN^k$ and all
$\lambda \in \Lambda$ with $d(\lambda) = m + n$, there exist unique $\mu,\nu \in \Lambda$ such that
$d(\mu) = m$, $d(\nu) = n$ and $\lambda = \mu\nu$.

We write $\Lambda^n$ for $d^{-1}(n)$. If $d(\lambda) = 0$ then $\lambda = \id_o$ for some object
$o$ of $\Lambda$. Hence $r(\lambda) := \id_{\operatorname{cod}(\lambda)}$ and $s(\lambda) :=
\id_{\operatorname{dom}(\lambda)}$ determine maps $r,s : \Lambda \to \Lambda^0$ which restrict to
the identity map on $\Lambda^0$ (see \cite{KP2000}). We think of elements of $\Lambda^0$ both as
vertices and as paths of degree 0, and we think of each $\lambda \in \Lambda$ as a path from
$s(\lambda)$ to $r(\lambda)$. If $v \in \Lambda^0$ and $\lambda \in \Lambda$, then the composition
$v\lambda$ makes sense if and only if $v = r(\lambda)$. With this in mind, given a subset $E$ of
$\Lambda$, and a vertex $v \in \Lambda^0$, we write $vE$ for the set $\{\lambda \in E : r(\lambda)
= v\}$. Similarly, $Ev$ denotes $\{\lambda \in E : s(\lambda) = v\}$. In particular, for $v \in
\Lambda^0$ and $n \in \NN^k$, we have $v\Lambda^n = \{\lambda \in \Lambda : d(\lambda) = n\text{
and } r(\lambda) = v\}$.  Moreover, given a subset $H$ of $\Lambda^0$, we let $EH$ denote the set
 $\{ \lambda \in E :  s(\lambda) \in H \}$ and set $HE = \{ \lambda \in E : r(\lambda) \in H \}$.

We say that $\Lambda$ is \emph{row-finite} if $v\Lambda^n$ is finite for all $v \in \Lambda^0$ and
$n \in \NN^k$. We say that $\Lambda$ has \emph{no sources} if $v\Lambda^n$ is nonempty for all $v
\in \Lambda^0$ and $n \in \NN^k$. We say that $\Lambda$ is \emph{locally convex} if, whenever $\mu
\in \Lambda^{e_i}$ and $r(\mu)\Lambda^{e_j} \not= \emptyset$ with $i \not= j$, we have
$s(\mu)\Lambda^{e_j} \not= \emptyset$ also.

For $\lambda \in \Lambda$ and $m \le n \le d(\lambda)$, we denote by $\lambda(m,n)$ the unique
element of $\Lambda^{n-m}$ such that $\lambda = \lambda'\lambda(m,n)\lambda''$ for some $\lambda',
\lambda''\in\Lambda$  with $d(\lambda') = m$ and $d(\lambda'') = d(\lambda) - n$.

For $\mu,\nu \in \Lambda$, a \emph{minimal common extension} of $\mu$ and $\nu$ is a path $\lambda$
such that $d(\lambda) = d(\mu) \vee d(\nu)$ and $\lambda = \mu\mu' = \nu\nu'$ for some $\mu', \nu'
\in \Lambda$. Equivalently, $\lambda$ is a minimal common extension of $\mu$ and $\nu$ if
$d(\lambda) = d(\mu) \vee d(\nu)$ and $\lambda(0, d(\mu)) = \mu$ and $\lambda(0, d(\nu)) = \nu$. We
write $\MCE(\mu,\nu)$ for the set of all minimal common extensions of $\mu$ and $\nu$, and we say
that $\Lambda$ is \emph{finitely aligned} if $\MCE(\mu,\nu)$ is finite (possibly empty) for all
$\mu,\nu \in \Lambda$. If $\Gamma$ is a sub-$k$-graph of $\Lambda$, then for $\mu,\nu \in \Gamma$
we write $\MCE_\Gamma(\mu,\nu)$ and $\MCE_\Lambda(\mu,\nu)$ to emphasise in which $k$-graph we are
computing the set of minimal common extensions. We have $\MCE_\Gamma(\mu,\nu) =
\MCE_\Lambda(\mu,\nu) \cap \Gamma \times \Gamma$.

For $\lambda \in \Lambda$ and $E \subseteq r(\lambda)\Lambda$, the set of paths $\tau \in
s(\lambda) \Lambda$ such that $\lambda\tau \in \MCE(\lambda,\mu)$ for some $\mu \in E$ is denoted
$\Ext(\lambda; E)$. That is,
\[
\Ext(\lambda; E) = \bigcup_{\mu \in E} \{\tau \in s(\lambda)\Lambda : \lambda\tau \in \MCE(\lambda,\mu)\}.
\]
By \cite[Proposition~3.12]{FMY2005}, we have $\Ext(\lambda\mu; E) = \Ext(\mu; \Ext(\lambda;E))$ for
all composable $\lambda,\mu$ and all $E \subseteq r(\lambda)\Lambda$.

Fix a vertex $v \in \Lambda^0$. A subset $F \subseteq v\Lambda$ is called \emph{exhaustive} if for
every $\lambda \in v\Lambda$ there exists $\mu \in F$ such that $\MCE(\lambda,\mu) \not=
\emptyset$.
 By \cite[Lemma~C.5]{RSY2004}, if $E \subset r(\lambda)\Lambda$ is exhaustive, then
$\Ext(\lambda;E) \subseteq s(\lambda)\Lambda$ is also exhaustive.

\subsection{Higher-rank graph \texorpdfstring{$C^*$}{C*}-algebras}
Let $\Lambda$ be a finitely aligned $k$-graph. A Cuntz-Krieger $\Lambda$-family is a subset
$\{t_\lambda : \lambda \in \Lambda\}$ of a $C^*$-algebra $B$ such that
\begin{enumerate}\renewcommand{\theenumi}{CK\arabic{enumi}}
    \item\label{it:CK1} $\{t_v : v \in \Lambda^0\}$ is a family of mutually orthogonal
        projections;
    \item\label{it:CK2} $t_\mu t_\nu = t_{\mu\nu}$ whenever $s(\mu) = r(\nu)$;
    \item\label{it:CK3} $t^*_\mu t_\nu = \sum_{\mu\alpha = \nu\beta \in \MCE(\mu,\nu)} t_\alpha
        t^*_\beta$ for all $\mu,\nu \in \Lambda$; and
    \item\label{it:CK4} $\prod_{\lambda \in E} (t_v -  t_\lambda t^*_\lambda) = 0$ for all $v
        \in \Lambda^0$ and finite exhaustive sets $E \subseteq v\Lambda$.
\end{enumerate}
The $C^*$-algebra $C^*(\Lambda)$ of $\Lambda$ is the universal $C^*$-algebra generated by a
Cuntz-Krieger $\Lambda$-family; the universal family in $C^*(\Lambda)$ is denoted $\{s_\lambda :
\lambda \in \Lambda\}$.

The universal property of $C^*(\Lambda)$ ensures that there exists a strongly continuous action
$\gamma$ of $\TT^k$ on $C^*(\Lambda)$ satisfying $\gamma_z(s_\lambda) = z^{d(\lambda)} s_\lambda$
for all $z \in \TT^k$ and $\lambda \in \Lambda$, where $z^{d(\lambda)}$ is defined by the standard
multi-index formula $z^{d(\lambda)} = z_1^{d(\lambda)_1} z_2^{d(\lambda)_2} \dots
z_k^{d(\lambda)_k}$.

The Cuntz-Krieger relations can be simplified significantly under additional hypotheses. For
details of the following, see \cite[Appendix~B]{RSY2004}. Suppose that $\Lambda$ is row-finite and
locally convex. For $n \in \NN^k$, define
\[
    \Lambda^{\le n} := \bigcup_{m \le n} \{\lambda \in \Lambda^m : s(\lambda)\Lambda^{e_i} = \emptyset
                                            \text{ for all }i \le k\text{ such that }m_i < n_i\}.
\]
Then (\ref{it:CK3})~and~(\ref{it:CK4}) are equivalent to
\begin{enumerate}\renewcommand{\theenumi}{$\text{CK\arabic{enumi}}'$}\setcounter{enumi}{2}
    \item\label{it:CK3'} $t^*_\mu t_\mu = t_{s(\mu)}$ for all $\mu \in \Lambda$, and
    \item\label{it:CK4'} $t_v = \sum_{\lambda \in v\Lambda^{\le n}} t_\lambda t^*_\lambda$ for
        all $v \in \Lambda^0$ and $n \in \NN^k$.
\end{enumerate}
If $\Lambda$ is has no sources, then $\Lambda^{\le n} = \Lambda^n$ for all $n$, so if $\Lambda$ is
row-finite and has no sources then~(\ref{it:CK4'}) is equivalent to
\begin{enumerate}\renewcommand{\theenumi}{$\text{CK\arabic{enumi}}''$}\setcounter{enumi}{3}
    \item\label{it:CK4''} $t_v = \sum_{\lambda \in v\Lambda^n} t_\lambda t^*_\lambda$ for all
        $v \in \Lambda^0$ and $n \in \NN^k$.
\end{enumerate}
Note that~(\ref{it:CK3}) implies~(\ref{it:CK3'}) for all $k$-graphs $\Lambda$.

Recall from \cite{PRS2008} that a \emph{graph trace} on a row-finite $k$-graph $\Lambda$ with no
sources is a function $g : \Lambda^0 \to \RR^+$ such that $g(v) = \sum_{\lambda \in v\Lambda^n}
g(s(\lambda))$ for all $v \in \Lambda^0$ and $n \in \NN^k$. A graph trace $g$ is called
\emph{faithful} if $g(v) \not= 0$ for all $v \in \Lambda^0$. Proposition~3.8 of \cite{PRS2008}
describes how faithful graph traces on $\Lambda$ correspond with faithful gauge-invariant
semifinite traces on $C^*(\Lambda)$. We call a graph trace $g$ \emph{finite} if $\sum_{v \in
\Lambda^0} g(v)$ converges to some $T \in \RR^+$, and we say that a finite graph trace $g$ is
\emph{normalised} if $\sum_{v \in \Lambda^0} g(v) = 1$.

\begin{lemma}\label{lem:graph traces revisited}
Let $\Lambda$ be a row-finite $k$-graph with no sources. Each normalised finite faithful graph
trace $g$ on $\Lambda$ determines a faithful bounded gauge-invariant trace $\tau_g$ on
$C^*(\Lambda)$ which is normalised in the sense that the limit over increasing finite subsets $F$
of $\Lambda^0$ of $\tau_g\Big(\sum_{v \in F} s_v\Big)$ is equal to $1$: specifically, $\tau_g(s_\mu
s^*_\nu) = \delta_{\mu,\nu}g(s(\mu))$ for all $\mu,\nu \in \Lambda$. Moreover, $g \mapsto \tau_g$
is a bijection between normalised finite faithful graph traces on $\Lambda$ and normalised faithful
gauge-invariant traces on $C^*(\Lambda)$.
\end{lemma}
\begin{proof}
By \cite[Proposition~3.8]{PRS2008}, the map $g \mapsto \tau_g$ is a bijection between faithful (not
necessarily finite or normalised) graph traces on $\Lambda$ and faithful semifinite
lower-semicontinuous gauge-invariant traces on $C^*(\Lambda)$. So it suffices to show that $\tau_g$
is finite if and only if $g$ is finite, and that $\tau_g$ is normalised if and only if $g$ is
normalised. For this, for each finite $F \subseteq \Lambda^0$ let $P_F := \sum_{v \in F} s_v \in
C^*(\Lambda)$. Then the $P_F$ form an approximate identity, and so $\tau_g$ is finite if and only
if $\lim_F \tau_g(P_F) = \sum_{v \in F} g(v)$ converges. Moreover, each of $g$ and $\tau_g$ is
normalised if and only if each of these sums converges to 1.
\end{proof}

\subsection{Infinite paths and aperiodicity}

For each $m \in (\NN \cup \{\infty\})^k$, we define a $k$-graph $\Omega_{k,m}$ by
\begin{gather*}
    \Omega_{k,m} = \{(p,q) \in \NN^k \times \NN^k : p \le q \le m\}, \text{ with} \\
    r(p,q) = (p,p),\qquad s(p,q) = (q,q),\qquad\text{ and }\qquad d(p,q) = q-p.
\end{gather*}
It is standard to identify $\Omega_{k,m}^0$ with $\{p \in \NN^k : p \le m\}$ by $(p,p) \mapsto p$,
and we shall silently do so henceforth.

If $\Lambda$ and $\Gamma$ are $k$-graphs, then a $k$-graph morphism $\phi : \Lambda \to \Gamma$ is
a functor from $\Lambda$ to $\Gamma$ which preserves degree: $d_\Gamma(\phi(\lambda)) =
d_\Lambda(\lambda)$ for all $\lambda \in \Lambda$.

Given a $k$-graph $\Lambda$ and $m \in \NN^k$, each $\lambda \in \Lambda^m$ determines a $k$-graph
morphism $x_\lambda : \Omega_{k, m} \to \Lambda$ by $x_\lambda(p,q) := \lambda(p,q)$ for all $(p,q)
\in \Omega_{k,m}$. Moreover, each $k$-graph morphism $x : \Omega_{k,m} \to \Lambda$ determines an
element $x(0,m)$ of $\Lambda^m$. Thus we identify the collection of $k$-graph morphisms from
$\Omega_{k,m}$ to $\Lambda$ with $\Lambda^m$ when $m \in \NN^k$. Extending this idea, given $m \in
(\NN \cup \{\infty\})^k \setminus \NN^k$, we regard $k$-graph morphisms $x : \Omega_{k,m} \to
\Lambda$ as paths of degree $m$ in $\Lambda$ and write $d(x):=m$ and $r(x)$ for $x(0)$; we denote
the set of all such paths by $\Lambda^m$. When $m = (\infty, \infty, \dots, \infty)$, we denote
$\Omega_{k,m}$ by $\Omega_k$ and we call a path $x$ of degree $m$ in $\Lambda$ an \emph{infinite
path.} We denote by $W_\Lambda$ the collection $\bigcup_{m \in (\NN \cup \{\infty\})^k} \Lambda^m$
of all paths in $\Lambda$; our conventions allow us to regard $\Lambda$ as a subset of $W_\Lambda$.

For each $n \in \NN^k$ there is a \emph{shift map} $\sigma^n : \{x \in W_\Lambda : n \le d(x)\} \to
W_\Lambda$ such that $d(\sigma^n(x)) = d(x) - n$ and $\sigma^n(x)(p,q) = x(n+p, n+q)$ for $0 \le p
\le q \le d(x)-n$. Given $x \in W_\Lambda$ and $\lambda \in \Lambda r(x)$, there is a unique
$\lambda x \in W_\Lambda$ satisfying $d(\lambda x) = d(\lambda) + d(x)$, $(\lambda x)(0,
d(\lambda)) = \lambda$ and $\sigma^{d(\lambda)}(\lambda x) = x$. For $x \in W_\Lambda$ and $n \le
d(x)$, we then have $x(0, n)\sigma^n(x) = x$.

A \emph{boundary path} in $\Lambda$ is a path $x : \Omega_{k,m} \to \Lambda$ with the property that
for all $p \in \Omega_{k,m}^0$ and all finite exhaustive sets $E \subseteq x(p)\Lambda$, there
exists $\mu \in E$ such that $x(p, p+d(\mu)) = \mu$. We denote by $\partial\Lambda$ the collection
of all boundary paths in $\Lambda$. Lemma~5.15 of~\cite{FMY2005} implies that for each $v \in
\Lambda^0$, the set $v\partial\Lambda := \{x \in \partial\Lambda : r(x) = v\}$ is nonempty. Fix $x
\in
\partial\Lambda$. If $n \le d(x)$, then $\sigma^n(x) \in \partial\Lambda$, and if $\lambda \in
\Lambda r(x)$, then $\lambda x \in \partial\Lambda$ \cite[Lemma~5.13]{FMY2005}. Recall also from
\cite{RSY2003} that if $\Lambda$ is row-finite and locally convex, then $\partial \Lambda$
coincides with the set
\[
    \Lambda^{\le\infty} = \{x \in W_\Lambda : x(n)\Lambda^{e_i} = \emptyset\text{ whenever }
        n \le d(x)\text{ and } n_i = d(x)_i\}.
\]

Recall from \cite{LS2010} that a $k$-graph $\Lambda$ is said to be \emph{aperiodic} if for all
$\mu,\nu \in \Lambda$ such that $s(\mu) = s(\nu)$ there exists $\tau \in s(\mu)\Lambda$ such that
$\MCE(\mu\tau, \nu\tau) = \emptyset$. By \cite[Proposition~3.6~and~Theorem~4.1]{LS2010}, the
following are equivalent:
\begin{enumerate}
    \item $\Lambda$ is aperiodic;
    \item for all distinct $m,n \in \NN^k$ and $v \in \Lambda^0$ there exists $x \in
        v\partial\Lambda$ such that either $m \vee n \not\le d(x)$ or $\sigma^m(x) \not=
        \sigma^n(x)$;
    \item for all $v \in \Lambda^0$ there exists $x \in v\partial\Lambda$ such that for
        distinct $m,n \le d(x)$, $\sigma^m(x) \not= \sigma^n(x)$;
    \item for every nontrivial ideal $I$ of $C^*(\Lambda)$ there exists $v \in \Lambda^0$ such
        that $s_v \in I$.
\end{enumerate}

Here, and in the rest of the paper, an ``ideal" of a $C^*$-algebra always means a closed 2-sided
ideal.

\subsection{Skeletons}

We will frequently wish to present a $k$-graph visually. To do this, we draw its \emph{skeleton}
and, if necessary, list the associated \emph{factorisation rules.}

Given a $k$-graph $\Lambda$, the \emph{skeleton} of $\Lambda$ is the coloured directed graph
$E_\Lambda$ with vertices $E_\Lambda^0 = \Lambda^0$, edges $E_\Lambda^1 := \bigcup^k_{i=1}
\Lambda^{e_i}$ and with colouring map $c : E_\Lambda^1 \to \{1, \dots, k\}$ given by $c(\alpha) =
i$ if and only if $\alpha \in \Lambda^{e_i}$. In pictures in this paper, edges of degree $e_1$ will
be drawn as solid lines and those of degree $e_2$ as dashed lines. If $\alpha, \beta \in
E_\Lambda^1$ have distinct colours, say $c(\alpha) = i$ and $c(\beta) = j$, and if $s(\alpha) =
r(\beta)$, then $\alpha\beta \in \Lambda^{e_i + e_j}$ and the factorisation property in $\Lambda$
implies that there are unique edges $\beta', \alpha' \in E_\Lambda^1$ such that $c(\beta')=
c(\beta)$ and $c(\alpha') = c(\alpha)$ and such that
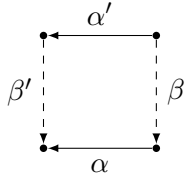
\[
\begin{tikzpicture}[scale=1.5]
    \node[circle, inner sep=1pt, fill=black] (00) at (0,0) {};
    \node[circle, inner sep=1pt, fill=black] (10) at (1,0) {};
    \node[circle, inner sep=1pt, fill=black] (01) at (0,1) {};
    \node[circle, inner sep=1pt, fill=black] (11) at (1,1) {};
    \draw[-latex] (11)--(01) node [pos=0.5, above] {\small$\alpha'$};
    \draw[-latex] (10)--(00) node [pos=0.5, below] {\small$\alpha$};
    \draw[-latex, dashed] (11)--(10) node [pos=0.5, right] {\small$\beta$};
    \draw[-latex, dashed] (01)--(00) node [pos=0.5, left] {\small$\beta'$};
\end{tikzpicture}
\]
is a commuting diagram in $\Lambda$. we call such a diagram a \emph{square} and we denote by
$\mathcal{C}$ the collection of all such squares. We write $\alpha\beta \sim_{\Cc} \beta'\alpha'$,
or just $\alpha\beta \sim \beta'\alpha'$. We call the list of all such relations the
\emph{factorisation rules} for $E_\Lambda$. It turns out that $\Lambda$ is uniquely determined up
to isomorphism by its skeleton and factorisation rules \cite{FS2002, HazelwoodRaeburnEtAl:xx11}.
Moreover, given a $k$-coloured directed graph $E$ and a collection of factorisation rules of the
form $\alpha\beta \sim \beta'\alpha'$ where $\alpha\beta$ and $\beta'\alpha'$ are bi-coloured paths
of opposite colourings with the same range and source, there exists a $k$-graph with this skeleton
and set of factorisation rules if and only if both of the following conditions are satisfied: (1)
the relation $\sim$ is bijective in the sense that for each $ij$-coloured path $\alpha\beta$, there
is exactly one $ji$-coloured path $\beta'\alpha'$ such that $\alpha\beta \sim \beta'\alpha'$; and
(2) if $\alpha\beta \sim \beta^1\alpha^1$, $\alpha^1\gamma \sim \gamma^1\alpha^2$ and
$\beta^1\gamma^1 \sim \gamma^2\beta^2$, and if $\beta\gamma \sim \gamma_1\beta_1$, $\alpha\gamma_1
\sim \gamma_2\alpha_1$ and $\alpha_1\beta_1 \sim \beta_2\alpha_2$, then $\alpha^2 = \alpha_2$,
$\beta^2 = \beta_2$ and $\gamma^2 = \gamma_2$. Observe that~(2) is vacuous unless $\alpha,\beta$
and $\gamma$ are of three distinct colours, so if $k = 2$, then condition~(1) by itself
characterises those lists of factorisation rules which determine $2$-graphs.

If $E_\Lambda$ has the property that given any two vertices $v,w$ and any two colours $i,j \le k$,
there is at most one path $fg$ from $w$ to $v$ such that $c(f) = i$ and $c(g) = j$, then there is
just one possible complete collection of squares possible for this skeleton. In this situation, we
just draw the skeleton to specify $\Lambda$, and do not bother to list the squares.

\section{Cycles and generalised cycles}

In this section we present a necessary condition on an arbitrary $k$-graph for its $C^*$-algebra to
be AF.

As with graph $C^*$-algebras, the necessary conditions for $k$-graph $C^*$-algebras to be AF which
we have developed involve the presence of cycles of an appropriate sort in the $k$-graph. To
formulate a result sufficiently general to deal with the examples which we introduce later, we
propose the notion of a \emph{generalised cycle}. We have not been able to construct a non-AF
$k$-graph $C^*$-algebra which could not be recognised as such by the presence of a generalised
cycle in the complement of some hereditary subgraph, but we have no reason to believe that such an
example does not exist. For the origins of the following definition, see~\cite[Lemma~4.3]{EvansPhD}

\begin{dfn}
Let $\Lambda$ be a finitely aligned $k$-graph. A \emph{generalised cycle} in $\Lambda$ is a pair
$(\mu,\nu) \in \Lambda \times \Lambda$ such that $\mu \not= \nu$, $s(\mu) = s(\nu)$, $r(\mu) =
r(\nu)$, and $\MCE(\mu\tau,\nu) \not= \emptyset$ for all $\tau \in s(\mu)\Lambda$.
\end{dfn}

\begin{lemma}\label{lem:gen cycle characterisations}
Let $\Lambda$ be a finitely aligned $k$-graph. Fix a pair $(\mu,\nu) \in \Lambda \times \Lambda$
such that $\mu \not= \nu$, $s(\mu) = s(\nu)$ and $r(\mu) = r(\nu)$. Then the following are
equivalent:
\begin{enumerate}
    \item\label{it:gen cyc} The pair $(\mu,\nu)$ is a generalised cycle;
    \item\label{it:Ext} The set $\Ext(\mu, \{\nu\})$ is exhaustive; and
    \item\label{it:boundary paths} $\{\mu x : x \in s(\mu)\partial\Lambda\} \subseteq \{\nu y :
        y \in s(\nu)\partial\Lambda\}$.
\end{enumerate}
\end{lemma}
\begin{proof}
Suppose that $(\mu,\nu)$ is a generalised cycle. Fix $\lambda \in s(\mu)\Lambda$. Then
$\MCE(\mu\lambda,\nu) \not= \emptyset$, and hence $\Ext(\mu\lambda, \{\nu\}) \not= \emptyset$. By
\cite[Proposition~3.12]{FMY2005}, we have
\[
    \Ext(\mu\lambda, \{\nu\}) = \Ext(\lambda; \Ext(\mu, \{\nu\})),
\]
and hence there exists $\alpha \in \Ext(\mu, \{\nu\})$ such that $\MCE(\lambda,\alpha) \not=
\emptyset$. Hence $\Ext(\mu, \{\nu\})$ is exhaustive. This proves \mbox{(\ref{it:gen
cyc})${}\implies{}$(\ref{it:Ext})}.

Now suppose that $\Ext(\mu, \{\nu\})$ is exhaustive. Since $\Lambda$ is finitely aligned,
$\Ext(\mu, \{\nu\})$ is also finite, and hence it is a finite exhaustive subset of $s(\mu)\Lambda$.
Fix $x \in s(\mu)\partial\Lambda$. By definition of $\partial\Lambda$ there exists $\alpha \in
\Ext(\mu,\{\nu\})$ such that $x(0, d(\alpha)) = \alpha$. Hence $(\mu x)(0, d(\mu) \vee d(\nu)) =
\mu\alpha \in \MCE(\mu,\nu)$, and it follows that $(\mu x)(0, d(\nu)) = (\mu\alpha)(0, d(\nu)) =
\nu$. Thus $y := \sigma^{d(\nu)}(\mu x)$ satisfies $y \in s(\nu) \partial\Lambda$ and $\mu x = \nu
y$. This proves \mbox{(\ref{it:Ext})${}\implies{}$(\ref{it:boundary paths})}.

Finally suppose that $\{\mu x : x \in s(\mu)\partial\Lambda\} \subseteq \{\nu y : y \in
s(\nu)\partial\Lambda\}$. Fix $\tau \in s(\mu)\Lambda$. Since $s(\tau) \partial\Lambda \not=
\emptyset$ \cite[Lemma~5.15]{FMY2005}, we may fix $z \in s(\tau)\partial\Lambda$, and then $x :=
\tau z \in s(\mu) \partial\Lambda$ also \cite[Lemma~5.13]{FMY2005}. By hypothesis, we then have
$\mu x = \nu y$ for some $y \in s(\nu)\partial\Lambda$. In particular, $(\mu x)(0, d(\mu\tau) \vee
d(\nu)) \in \MCE(\mu\tau, \nu)$, and hence the latter is nonempty. This proves
\mbox{(\ref{it:boundary paths})${}\implies{}$(\ref{it:gen cyc})}.
\end{proof}

In the language of \cite{FMY2005}, condition~(\ref{it:boundary paths}) of Lemma~\ref{lem:gen cycle
characterisations} says that the cylinder sets $Z(\mu)$ and $Z(\nu)$ are nested: $Z(\mu) \subseteq
Z(\nu)$.

For the remainder of the paper, the term \emph{cycle}, as distinct from \emph{generalised cycle},
will continue to refer to a path $\lambda \in \Lambda\setminus\Lambda^0$ such that $r(\lambda) =
s(\lambda)$. When
--- as in Section~\ref{sec:unital} --- we wish to emphasise that we mean a cycle in the traditional
sense, rather than a generalised cycle, we will also use the term \emph{conventional cycle.}

To see where the definition of a generalised cycle comes from, observe that if $\lambda$ is a
conventional cycle in a $k$-graph, then $(\lambda, r(\lambda))$ is a generalised cycle. There are
plenty of examples of $k$-graphs containing generalised cycles but no cycles (see
Example~\ref{non-AF aper}), but when $k = 1$, the two notions more or less coincide:

\begin{lemma}\label{lem:cycle discussion}
Let $\Lambda$ be a $1$-graph. Suppose that $(\mu,\nu)$ is a generalised cycle in $\Lambda$. Then
there is a conventional cycle $\lambda \in \Lambda \setminus \Lambda^0$ such that either $\mu =
\nu\lambda$ or $\nu = \mu\lambda$.
\end{lemma}
\begin{proof}
Since $\Lambda$ is a $1$-graph, either $d(\mu) \le d(\nu)$ or vice versa. We will assume that
$d(\mu) \le d(\nu)$ and show that $\nu = \mu\lambda$ for some conventional cycle $\lambda$; if
instead $d(\nu) \le d(\mu)$ then the same argument gives $\nu = \mu\lambda$. If $d(\mu) = d(\nu)$,
then $\MCE(\mu,\nu) \not= \emptyset$ forces $\mu = \nu$ which is impossible for a generalised
cycle, so $d(\mu) < d(\nu)$. Then $\tau := s(\mu) \in s(\mu)\Lambda$ satisfies $\MCE(\mu\tau, \nu)
\not=\emptyset$. This forces $\nu = \mu\lambda$ for some $\lambda$. Now $r(\lambda) = s(\mu)$ and
$s(\lambda) = s(\nu) = s(\mu)$, so $\lambda$ is a conventional cycle.
\end{proof}

The main result in this section is the following.

\begin{theorem}\label{thm:main necessary}
Let $\Lambda$ be a finitely aligned $k$-graph. If $C^*(\Lambda)$ is AF, then $\Lambda$ contains no
generalised cycles.
\end{theorem}

The proof deals separately with two cases. To delineate the cases, we introduce the notion of an
entrance to a generalised cycle.

\begin{dfn}
Let $\Lambda$ be a finitely aligned $k$-graph. An \emph{entrance} to a generalised cycle
$(\mu,\nu)$ is a path $\tau \in s(\nu)\Lambda$ such that $\MCE(\nu\tau, \mu) = \emptyset$.
\end{dfn}

If $\lambda$ is a conventional cycle then an \emph{entrance to the conventional cycle} $\lambda$
means an entrance to the associated generalised cycle $(\lambda, r(\lambda))$; that is a path $\tau
\in r(\lambda) \Lambda$ such that $\MCE(\tau, \lambda) = \emptyset$.

\begin{rmk}\label{rmk:FMY lemma}
A generalised cycle $(\mu,\nu)$ has an entrance if and only if the reversed pair $(\nu,\mu)$ is not
a generalised cycle.
\end{rmk}

\begin{lemma}\label{lem:compare projs}
Suppose that $(\mu,\nu)$ is a generalised cycle. Then $s_\mu s_\mu^* \le s_\nu s_\nu^*$. Moreover,
$s_\mu s_\mu^* = s_\nu s_\nu^*$ if and only if the generalised cycle $(\mu,\nu)$ has no entrance.
\end{lemma}
\begin{proof}
Since $\Ext(\mu, \{\nu\}) \subset s(\mu)\Lambda^{(d(\mu) \vee d(\nu)) - d(\mu)}$, for distinct
$\alpha,\beta \in \Ext(\mu,\{\nu\})$, we have $s_\alpha s^*_\alpha s_\beta s^*_\beta = 0$. In
particular, applying~(\ref{it:CK4}),
\[
0   = \prod_{\alpha \in \Ext(\mu,\{\nu\})} (s_{s(\mu)} - s_\alpha s^*_\alpha)
    = s_{s(\mu)} - \sum_{\alpha \in \Ext(\mu,\{\nu\})} s_\alpha s^*_\alpha.
\]
Hence
\[
s_\mu s^*_\mu
    = s_\mu s_{s(\mu)} s^*_\mu
    = \sum_{\alpha \in \Ext(\mu,\{\nu\})} s_{\mu\alpha} s^*_{\mu\alpha}.
\]
For each $\alpha \in \Ext(\mu,\{\nu\})$, we have $\mu\alpha = \nu\beta$ for some $\beta \in
\Lambda$, and hence $s_{\mu\alpha} s^*_{\mu\alpha} = s_\nu (s_\beta s^*_\beta) s^*_{\nu} \le s_\nu
s^*_\nu$, giving $s_\mu s^*_\mu \le s_\nu s^*_\nu$.

Suppose that the generalised cycle $(\mu,\nu)$ has no entrance. Then $(\nu,\mu)$ is also a
generalised cycle, and the preceding paragraph gives $s_\nu s^*_\nu \le s_\mu s^*_\mu$ also.

Now suppose that the generalised cycle $(\mu,\nu)$ has an entrance $\tau$; so $\MCE(\nu\tau,\mu) =
\emptyset$. Then $s_{\nu\tau} s^*_{\nu\tau} \le s_\nu s^*_\nu$ and
\[
s_{\nu\tau} s^*_{\nu\tau} s_\mu s^*_\mu
    = \sum_{\lambda \in \MCE(\nu\tau,\mu)} s_\lambda s^*_\lambda = 0.
\]
Hence $s_\nu s^*_\nu - s_\mu s^*_\mu \ge s_{\nu\tau} s^*_{\nu\tau} > 0$.
\end{proof}

\begin{cor}[{\cite[Lemma~4.3]{EvansPhD}}]\label{cor:inf proj}
Let $\Lambda$ be a finitely aligned $k$-graph which contains a generalised cycle with an entrance.
Then $C^*(\Lambda)$ contains an infinite projection. In particular $C^*(\Lambda)$ is not AF.
\end{cor}
\begin{proof}
Let $(\mu,\nu)$ be the generalised cycle with an entrance. By Lemma~\ref{lem:compare projs}, we
have
\[
s_\nu s^*_\nu
    > s_\mu s^*_\mu
    = s_\mu s^*_\nu s_\nu s^*_\mu
    \sim s_\nu s^*_\mu s_\mu s^*_\nu
    = s_\nu s^*_\nu.
\]
Hence $s_\nu s^*_\nu$ is an infinite projection. The last statement follows immediately.
\end{proof}

We must now show that when $\Lambda$ contains a generalised cycle with no entrance, $C^*(\Lambda)$
is not AF. The following result is the key step. The argument is essentially that of
\cite[Proposition~4.4.1]{BEH1980}, and we thank George Elliott for directing our attention to
\cite{BEH1980}.

\begin{prop}\label{prp:gauge unitaries}
Let $A$ be a unital $C^*$-algebra carrying a normalised trace $T$, and let $\beta : \TT \to
\Aut(A)$ be a strongly continuous action. Let $U$ be a unitary in $A$, and suppose that there
exists $n \in \ZZ\setminus\{0\}$ satisfying $\beta_z(U) = z^n U$ for all $z \in \TT$. Then $U$ does
not belong to the connected component of the identity in the unitary group $\mathcal{U}(A)$.
\end{prop}
\begin{proof}
Let $\alpha : \mathbb{R} \to \Aut(A)$ be the action determined by $\alpha_t(a) := \beta_{e^{2\pi\im
t}}(a)$. Let $\Dd(\delta) := \big\{a \in A : \lim_{t \to 0} \frac{1}{t} (\alpha_t(a) - a)\text{
exists}\big\}$, and let $\delta : \Dd(\delta) \to A$ be the generator of $\alpha$; that is
$\delta(a) := \lim_{t \to 0} \frac{1}{t} (\alpha_t(a) - a)$ for $a \in \Dd(\delta)$. Note that
$U\in\Dd(\delta)$ since we have
\begin{equation}\label{eq:delta(U)}
\delta(U)
    = \lim_{t \to 0} \frac{1}{t} (\beta_{e^{2\pi\im t}}(U) - U)
    = \lim_{t \to 0} \frac{e^{2n\pi\im t} -1}{t} U
    = 2n\pi\im U.
\end{equation}
Let $\mu$ denote the normalised Haar measure on $\TT$. Define a map $\tau : A \to \CC$ by $\tau(a)
:= \int_{\TT} T(\beta_z(a))\,d\mu(z)$. We claim that $\tau$ is a normalised $\beta$-invariant (and
hence $\alpha$-invariant) trace on $A$. Given $a \in A$, for each $z \in \TT$ we have
$T(\beta_z(a^*a)) = T(\beta_z(a)^*\beta_z(a)) \ge 0$ as $T$ is a trace. Hence $\tau(a^* a) \ge 0$
so $\tau$ is positive. It is clearly linear, and it satisfies $\tau(1) = 1$ because $\beta$ fixes
1. For $a,b \in A$ we calculate:
\[
\tau(ab)
    = \int_\TT T(\beta_z(a)\beta_z(b))\,d\mu(z)
    = \int_\TT T(\beta_z(b)\beta_z(a))\,d\mu(z)
    = \tau(ba).
\]
So $\tau$ is a trace. Finally, to see that $\tau$ is $\beta$-invariant, note that for $a \in A$, we
have $\tau(\beta_z(a)) = \int_{\TT} T(\beta_w(\beta_z(a)))\,d\mu(w) = \int_\TT
T(\beta_{zw}(a))\,d\mu(w) = \int_\TT \beta_{w'}(a) d\mu(z^{-1}w') = \tau(a)$ by left-invariance of
$\mu$.

It now follows from \cite[p~281, lines 7--16]{PW1978} that for a unitary $V \in \Dd(\delta)$ which
is also  in the connected component $\mathcal{U}_0(A)$ of the identity, we have $\tau(V^*\delta(V))
= 0$.  However, using~\eqref{eq:delta(U)}, we have $\tau(U^*\delta(U)) = \tau(U^* 2n\pi\im U) =
\tau(2n\pi\im 1_A) = 2n\pi\im$, and it follows that $U \not\in \mathcal{U}_0(A)$.
\end{proof}

\begin{prop}\label{prp:arbitrary partial unitary}
Let $\Lambda$ be a finitely aligned $k$-graph, and let $\phi : \ZZ^k \to \ZZ$ be a homomorphism.
Suppose that there exists $N \in \ZZ\setminus\{0\}$ and a partial isometry $V \in \clsp\{s_\mu
s^*_\nu : \mu,\nu \in \Lambda, \phi(d(\mu) - d(\nu)) = N\}$ such that $VV^* = V^*V$. Then
$C^*(\Lambda)$ is not AF.
\end{prop}
\begin{proof}
Let $P = V^*V$. Then $V$ is a unitary in $P C^*(\Lambda) P$.

For each $i \le k$, let $\phi_i = \phi(e_i)$ so that $\phi(n) = \sum^k_{i=1} \phi_i n_i$ for all $n
\in \ZZ^k$. Define a homomorphism $\iota_\phi : \TT \to \TT^k$ by $\iota(z)_i = z^{\phi_i}$ for $1
\le i \le k$, and define $\beta : \TT \to \Aut(C^*(\Lambda))$ by $\beta_z :=
\gamma_{\iota_\phi(z)}$ for all $z \in \TT$.

For $\mu,\nu \in \Lambda$ we have
\[
\beta_z(s_\mu s^*_\nu)
= \gamma_{\iota_\phi(z)}(s_\mu s^*_\nu)
= \iota_\phi(z)^{d(\mu) - d(\nu)} s_\mu s^*_\nu
= z^{\phi(d(\mu) - d(\nu))} s_\mu s^*_\nu.
\]
In particular, since $V \in \clsp\{s_\mu s^*_\nu : \phi(d(\mu) - d(\nu)) = N\}$, we have
$\beta_z(V) = z^N V$ for all $z \in \TT$ so that $\beta$ fixes $P$ and hence restricts to an action
on $P C^*(\Lambda) P$. Now suppose that $C^*(\Lambda)$ is an AF algebra; we seek a contradiction.
Since corners of AF algebras are AF \cite[Exercise~III.2]{Davidson1996}, $P C^*(\Lambda) P$ is a
unital AF algebra, and hence carries a normalised trace. We may therefore apply
Proposition~\ref{prp:gauge unitaries} to see that $V$ does not belong to the connected component of
the unitary group of $P C^*(\Lambda) P$. This is a contradiction since the unitary group of any
unital AF algebra is connected.
\end{proof}

\begin{proof}[Proof of Theorem~\ref{thm:main necessary}]
We prove the contrapositive statement. Let $(\mu,\nu)$ be a generalised cycle in $\Lambda$. If
$(\mu,\nu)$ has an entrance, then Corollary~\ref{cor:inf proj} implies that $C^*(\Lambda)$ is not
AF. So suppose that $(\mu,\nu)$ has no entrance.

Since $d(\mu) \not= d(\nu)$ there exists $i$ such that $d(\mu)_i \not= d(\nu)_i$. Define $\phi :
\ZZ^k \to \ZZ$ by $\phi(n) := n_i$, let $N := d(\mu)_i - d(\nu)_i$, and let $V := s_\mu s^*_\nu$.
By Lemma~\ref{lem:compare projs}, we have $VV^* = V^*V$, so Proposition~\ref{prp:arbitrary partial
unitary} applied to $V, N, \phi$ implies that $C^*(\Lambda)$ is not AF.
\end{proof}

Using the characterisation of gauge-invariant ideals in $k$-graph algebras of \cite{Sims2006a}, and
using also that quotients of AF algebras are AF, we can extend the main theorem somewhat, at the
expense of a more technical statement. Example~\ref{eg:spine} indicates that the extended result is
genuinely stronger.

\begin{cor}\label{cor:GCs in quotients}
Let $\Lambda$ be a finitely aligned $k$-graph. Suppose that there exists a saturated hereditary
subset $H$ of $\Lambda^0$ such that $\Lambda \setminus \Lambda H$ contains a generalised cycle.
Then $C^*(\Lambda)$ is not AF.

Moreover, given a saturated hereditary subset $H$ of $\Lambda^0$, a pair $(\mu,\nu) \in (\Lambda
\setminus \Lambda H)^2$ is a generalised cycle in $\Lambda \setminus \Lambda H$ if and only if
$d(\mu) \not= d(\nu)$, $s(\mu) = s(\nu)$, $r(\mu) = r(\nu)$, and $\MCE_\Lambda(\nu, \mu\tau)
\setminus \Lambda H \not=\emptyset$ for every $\tau \in \Lambda\setminus\Lambda H$.
\end{cor}
\begin{proof}
For the first statement observe that by \cite[Lemma~4.1]{Sims2006a}, $\Lambda \setminus \Lambda H$
is a finitely aligned $k$-graph, and \cite[Corollary~5.3]{Sims2006a} applied with $B =
\operatorname{FE}(\Lambda \setminus \Lambda H) \setminus \mathcal{E}_H$ implies that $C^*(\Lambda
\setminus \Lambda H)$ is a quotient of $C^*(\Lambda)$. If $\Lambda \setminus \Lambda H$ contains a
generalised cycle, then Theorem~\ref{thm:main necessary} implies that $C^*(\Lambda \setminus
\Lambda H)$ is not AF, and since quotients of AF algebras are AF, it follows that $C^*(\Lambda)$ is
not AF either.

For the final statement, observe that
\[
\MCE_{\Lambda \setminus \Lambda H}(\alpha,\beta) = \MCE_\Lambda(\alpha,\beta) \setminus \Lambda H.
\]
So by Remark~\ref{rmk:FMY lemma}, a generalised cycle in $\Lambda \setminus \Lambda H$ is a pair of
distinct paths $(\mu,\nu)$ in $\Lambda \setminus \Lambda H$ with the same range and source such
that for every $\tau \in s(\mu) \Lambda \setminus \Lambda H$, the set $\MCE_\Lambda(\nu,\mu\tau)
\setminus \Lambda H$ is nonempty as claimed.
\end{proof}

Theorem~\ref{thm:main necessary} combined with the results of \cite{LS2010} shows in particular
that aperiodicity of every quotient graph is necessary for $C^*(\Lambda)$ to be AF. We use this to
show that if $C^*(\Lambda)$ is AF, then its ideals are indexed by the saturated hereditary subsets
of $\Lambda^0$.

\begin{prop}
Let $\Lambda$ be a finitely aligned $k$-graph such that $C^*(\Lambda)$ is AF. Then for every
saturated hereditary subset $H$ of $\Lambda$, and every pair $\eta,\zeta$ of distinct paths in
$\Lambda \setminus \Lambda H$, there exists $\tau \in s(\eta)\Lambda \setminus \Lambda H$ such that
$\MCE(\eta\tau,\zeta\tau) \subset \Lambda H$. Moreover every ideal of $C^*(\Lambda)$ is
gauge-invariant.
\end{prop}
\begin{proof}
For the first statement of the Proposition, we prove the contrapositive. Suppose that there exist a
saturated hereditary $H \subset \Lambda^0$ and distinct paths $\eta,\zeta \in \Lambda \setminus
\Lambda H$ such that for every $\tau \in s(\eta)\Lambda \setminus \Lambda H$, we have
$\MCE(\eta\tau, \zeta\tau) \cap (\Lambda \setminus \Lambda H) \not= \emptyset$. Let $\Gamma :=
\Lambda \setminus \Lambda H$. The paths $\eta,\zeta \in \Gamma$ have the property that for every
$\tau \in s(\eta)\Gamma$, we have $\MCE_\Gamma(\eta\tau,\zeta\tau) \not= \emptyset$. That is,
$\Gamma$ is not aperiodic in the sense of \cite[Definition~3.1]{LS2010}.

By \cite[Proposition~3.6 and Definition~3.5]{LS2010}, there exist $v \in \Gamma^0$ and distinct
$m,n \in \NN^k$ such that $m \vee n \le d(x)$ and $\sigma^m(x) = \sigma^n(x)$ for all $x \in
v\partial\Gamma$. By \cite[Lemma~4.3]{LS2010}, there then exist $\mu,\nu,\alpha \in \Gamma$ such
that $d(\mu) = m$, $d(\nu) = n$, $r(\mu) = r(\nu)$, $s(\mu) = s(\nu) = r(\alpha)$, and $\mu\alpha x
= \nu\alpha x$ for all $x \in s(\alpha)\partial\Gamma$. In particular $\{\mu\alpha x : x \in
s(\mu\alpha)\partial\Lambda\} \subseteq \{\nu\alpha y : y \in s(\nu\alpha)\partial \Lambda\}$. So
\mbox{(\ref{it:boundary paths})${}\implies{}$(\ref{it:gen cyc})} of Lemma~\ref{lem:gen cycle
characterisations} implies that $(\mu\alpha, \nu\alpha)$ is a generalised cycle in $\Gamma$, and
then Theorem~\ref{thm:main necessary} implies that $C^*(\Gamma)$ is not AF.

Corollary~5.3 of \cite{Sims2006a} implies that $C^*(\Gamma)$ is isomorphic to the quotient of
$C^*(\Lambda)$ by the ideal generated by $\{s_v : v \in H\}$. Since quotients of AF algebras are
AF, it follows that $C^*(\Lambda)$ is also not AF.

To prove the second statement, suppose that $C^*(\Lambda)$ is indeed AF. The previous statement
combined with \cite[Lemma~4.4]{LS2010} implies that for each saturated hereditary $H \subseteq
\Lambda^0$, each $v \in \Lambda^0 \setminus H$ and each finite $F \subset \Lambda v$, there exists
$\tau \in v\Lambda \setminus \Lambda H$ such that $\MCE(\mu\tau, \nu\tau) = \emptyset$ for all
distinct $\mu,\nu \in F$. We may now run the proof of \cite[Theorem~6.3]{Sims2006}, leaving out
Lemma~6.4 and the first two paragraphs of the proof of Lemma~6.7 and using $\tau$ in place of the
path $x(0, N)$ in the remainder of the proof of Lemma~6.7, to see that the conclusion of
Theorem~6.3 holds for any relative Cuntz-Krieger algebra associated to $\Lambda \setminus \Lambda
H$; and then the argument of \cite[Theorem~7.2]{Sims2006a} implies that every ideal of
$C^*(\Lambda)$ is gauge-invariant as claimed.
\end{proof}

\section{Corners and skew-products}

We begin this section with a characterisation of approximate finite-dimensionality of
$C^*(\Lambda)$ in terms of the same property for corners of the form $s_v C^*(\Lambda) s_v$. We
then describe a recipe for constructing examples of $k$-graphs whose $C^*$-algebras are AF.

\begin{prop} \label{prp:vertex corners}
Let $(\Lambda,d)$ be a finitely aligned $k$-graph. Then $C^*(\Lambda)$ is AF if and only if the
corners $s_v C^*(\Lambda) s_v$, $v \in \Lambda^0$ are all AF.
\end{prop}
\begin{proof}
It is standard that corners of AF algebras are AF (see, for example,
\cite[Exercise~III.2]{Davidson1996}), proving the ``only if'' implication.

For the ``if'' direction, suppose that each $s_v C^*(\Lambda) s_v$ is AF. For each finite $F
\subset \Lambda^0$, let $P_F := \sum_{v \in F} s_v$. We claim that each ideal $I_F := C^*(\Lambda)
P_F C^*(\Lambda)$ is AF. We proceed by induction on $|F|$. If $|F| = 1$, say $F = \{v\}$, then $I_F
\Me s_v C^*(\Lambda) s_v$ is AF by hypothesis. Now suppose that $I_F$ is AF whenever $|F| \le n$
and fix $F \subset \Lambda^0$ with $|F| = n+1$. Fix $v \in F$, and let $G = F \setminus \{v\}$.
Then $I_G$ and $I_{\{v\}}$ are both AF by the inductive hypothesis, and hence $I_G/(I_{\{v\}} \cap
I_G)$ is also AF because quotients of AF algebras are AF. Since $I_F = I_G + I_{\{v\}}$, there is
an exact sequence
\[
I_{\{v\}} \to I_F \to I_G/(I_{\{v\}} \cap I_G).
\]
Since extensions of AF algebras by AF algebras are also AF (see, for example
\cite[Theorem~III.6.3]{Davidson1996}), it follows that $I_F$ is AF as claimed.

Clearly $G \subseteq F$ implies $I_G \subseteq I_F$. Thus $C^*(\Lambda) = \overline{\bigcup_{F
\subset \Lambda^0\text{ finite}} I_F}$ is AF because the class of AF algebras is closed under
taking countable direct limits (this follows from an $\varepsilon/2$ argument using
\cite[Theorem~2.2]{Bratteli1972a}).
\end{proof}

We now describe a class of examples of $k$-graphs whose $C^*$-algebras are AF. The primary
motivation is the example $\LAF$ discussed in Section~\ref{sec:examples}.

Recall from \cite[Definition~1.9]{KP2000} that if $f : \NN^k \to \NN^l$ is a surjective
homomorphism, and $\Lambda$ is an $l$-graph, then there is a pullback $k$-graph $f^*(\Lambda)$
equal as a set to $\{(\lambda, n) \in \Lambda \times \NN^k : d(\lambda) = f(n)\}$ with pointwise
operations and degree map $d(\lambda,n) := n$. Also recall that if $c : \Lambda \to \ZZ^k$ is a
functor from a $k$-graph to $\ZZ^k$, then we can form the skew-product $k$-graph $\Lambda \times_c
\ZZ^k$, which is equal as a set to $\Lambda \times \ZZ^k$ with structure maps $r(\lambda,m) =
(r(\lambda), m)$, $s(\lambda,m) = (s(\lambda), m + c(\lambda))$, and $(\lambda, m)(\mu, m +
c(\lambda)) = (\lambda\mu, m)$.

\begin{example}\label{eg:skew-pullback}
Let $E$ be a row-finite  $1$-graph, and denote the degree functor on $E$ by $|\cdot| : E \to \NN$.
Let $c_0$ be a function from $E^1$ to $\{0, e_1, \dots, e_{k-1}\} \subseteq \ZZ^{k}$, and for
$\lambda = \lambda_1\cdots\lambda_n \in E^n$, let $c_0(\lambda) := \sum^n_{i=1} c_0(\lambda_i)$.
Define $c_0(v) = 0$ for $v \in E^0$. Define $f : \NN^k \to \NN$ by $f(n) = \sum^k_{i=1} n_i$, and a
functor $c : f^*(E) \to \ZZ^k$ by
\[
c(\lambda,n)_j := \begin{cases}
    c_0(\lambda)_j - \sum_{i \not= j} n_i &\text{ if $j < k$,} \\
    |\lambda| &\text{ if $j = k$.} \\
\end{cases}
\]
Let $\Lambda$ be the skew-product $k$-graph $\Lambda = f^*(E)\times_c \ZZ^k$. Identifying $(E
\times \NN^k) \times \NN^k$ with $E \times \NN^k \times \NN^k$, we have
\[
    f^* (E) \times_c \ZZ^k = \{(\alpha, m, q) \in E \times \NN^k \times \ZZ^k : |\alpha| = f(m)\}.
\]
Observe that $p(\lambda, n, a) := \lambda$ defines a functor from $\Lambda$ to $E$. In particular,
each $v \in \Lambda^0$ has the form $v = (p(v), 0, q)$ for some $q \in \ZZ^k$, and then $\mu =
(p(\mu), d(\mu), q)$ for all $\mu \in v\Lambda$.

We will show that $C^*(\Lambda)$ is AF, with the corners of $C^*(\Lambda)$ determined by vertex
projections isomorphic to corresponding corners of the AF core of $C^*(E)$.
\end{example}

\begin{lemma}\label{lem:equal pis}
Consider the situation of Example~\ref{eg:skew-pullback}. Fix a vertex $v = (p(v), 0, q) \in
\Lambda^0$. Fix $\mu,\nu \in v\Lambda$, let $m := d(\mu)$ and $n := d(\nu)$, and express $s(\mu)$
as $(w, 0, q) \in E^0 \times \{0\} \times \ZZ^k = \Lambda^0$. Then
\[
s^*_\mu s_\nu
    = \begin{cases}
        \sum_{\tau \in p(s(\mu))E^{|n|}} s_{(\tau, n, q + c(\mu,m))} s^*_{(\eta\tau, m, q + c(\nu,n))}
            &\text{ if $p(\mu) = p(\nu)\eta$,} \\
        \sum_{\tau \in p(s(\nu))E^{|m|}} s_{(\zeta\tau, n, q + c(\mu,m))} s^*_{(\tau, m, q + c(\nu,n))}
            &\text{ if $p(\nu) = p(\mu)\zeta$,} \\
        0 &\text{ otherwise.}
    \end{cases}
\]
\end{lemma}
\begin{proof}
Let $N := |\mu| + |\nu|$. We have
\begin{align*}
s^*_\mu s_\nu
    &= \sum_{\lambda \in v\Lambda^{m+n}} s^*_\mu s_\lambda s^*_\lambda s_\nu \nonumber\\
    &= \sum_{\xi \in wE^N} s^*_{(p(\mu), m, q)} s_{(\xi,m+n,q)} s^*_{(\xi,m+n,q)} s_{(p(\nu), n, q)}.
\end{align*}
Factorising each $\xi \in wE^N$ as $\xi = \xi_m\xi'_m = \xi_n\xi'_n$ where $|\xi_m| = |m|$ and
$|\xi_n| = |n|$ gives
\begin{equation}
s^*_\mu s_\nu
    = \sum_{\xi \in wE^N} (s^*_{(p(\mu), m, q)} s_{(\xi_m,m,q)}) s_{(\xi'_m, n, q + c(\xi_m,m))}
            s^*_{(\xi'_n, m, q + c(\xi_n,n))} (s^*_{(\xi_n,n,q)} s_{(p(\nu), n, q)}).\label{eq:pull star across}
\end{equation}
The Cuntz-Krieger relations ensure that $s^*_{(p(\mu), m, q)} s_{(\xi_m,m,q)} = 0$ unless $\xi_m =
p(\mu)$, and likewise $s^*_{(\xi_n,n,q)} s_{(p(\nu), n, q)} = 0$ unless $\xi_n = p(\nu)$. So
$s^*_\mu s_\nu = 0$ unless there exists $\xi \in E^N$ such that $\xi = p(\mu)\xi'_m = p(\nu)\xi'_n$
which occurs if and only if either: (1) $p(\mu) = p(\nu)\eta$, and $\xi'_m = \tau$ and $\xi'_n =
\eta\tau$ for some $\tau \in E^{|n|}$; or (2) $p(\nu) = p(\mu)\zeta$, and $\xi'_n = \tau$ and
$\xi'_m = \zeta\tau$ for some $\tau \in E^{|m|}$. Using this to simplify~\eqref{eq:pull star
across}, we obtain the desired formula for $s^*_\mu s_\nu$.
\end{proof}

\begin{lemma}\label{lem:a,b exist}
Consider the situation of Example~\ref{eg:skew-pullback}. Fix $\alpha,\beta \in E^N$. Define $a,b
\in \NN^k$ by
\[
    a_j = \begin{cases}
                c_0(\beta)_j &\text{ $j < k$} \\
                N - |c_0(\beta)| &\text{ $j = k$}
            \end{cases} \qquad\text{ and }\qquad
    b_j = \begin{cases}
                c_0(\alpha)_j + c_j &\text{ $j < k$,} \\
                N - |c_0(\alpha)| &\text{ $j = k$.}
            \end{cases}
\]
Then $(\alpha,a), (\beta,b) \in f^* (E)$. If $s(\alpha) = s(\beta)$, then $s(\alpha,a,q) =
s(\beta,b,q)$ for all $q \in \NN^k$, and if $s(\alpha) \not= s(\beta)$, then $s(\alpha,a,q) \not=
s(\beta,b,q)$ for all $q \in \NN^k$.
\end{lemma}
\begin{proof}
Clearly $f(a) = f(b) = N = |\alpha| = |\beta|$, so $(\alpha,a), (\beta,b) \in f^* (E)$. For $j
\not= k$, we have
\[
c(\alpha,a)_j = c_0(\alpha)_j - \sum_{i \not= j} a_i = b_j - (N - a_j) = N + a_j + b_j,
\]
and similarly $c(\beta,b)_j = N + b_j + a_j$. Since $c(\alpha,a)_k = N = c(\beta,b)_k$, we have
$c(\alpha,a) = c(\beta,b)$ and the result follows.
\end{proof}

\begin{cor}\label{cor:skew-pullback-structure}
Consider the situation of Example~\ref{eg:skew-pullback}. For $v \in \Lambda^0$ and $N \in \NN$,
let $B_N(v) := \lsp\{s_\mu s^*_\nu : \mu,\nu \in v\Lambda, |\mu| = |\nu| = N\}$. Then for each $N
\in \NN$, the space $B_N(v)$ is a finite dimensional $C^*$-algebra with nonzero matrix units
\[
    \{w_{\eta, \zeta} : \eta,\zeta \in p(v)E^N, s(\eta) = s(\zeta)\}.
\]
Moreover $s_\mu s^*_\nu = \sum_{\xi \in s(\mu)\Lambda^{e_1}} s_{\mu \xi} s^*_{\nu\xi}$ determines
an inclusion $B_N(v) \subseteq B_{N+1}(v)$, and $s_v C^*(\Lambda) s_v = \overline{\bigcup_{N =
1}^\infty B_N(v)}$. For each $v \in \Lambda^0$, we have $s_v C^*(\Lambda) s_v \cong t_{p(v)}
C^*(E)^\gamma t_{p(v)}$. If $\Lambda$ is cofinal, then $C^*(\Lambda)$ is strongly Morita equivalent
to $s_w C^*(E)^\gamma s_w$ for any $w \in E^0$.
\end{cor}
\begin{proof}
We first claim that if $\mu,\nu,\alpha,\beta \in v\Lambda$ with $|\mu| = |\nu| = |\alpha| = |\beta|
= N$, $s(\mu) = s(\nu)$, $s(\alpha) = s(\beta)$, $p(\mu) = p(\alpha)$ and $p(\nu) = p(\beta)$, then
$s_\mu s^*_\nu = s_\alpha s^*_\beta$.

To see this, let $m:= d(\mu), n := d(\nu), a := d(\alpha)$, and $b := d(\beta)$, and use
Lemma~\ref{lem:equal pis} to calculate
\begin{align}
s_\mu s^*_\nu
    &= s_\mu s^*_\nu \sum_{\sigma \in p(v)E^{2N}} s_{(\sigma, n + b, q)} s^*_{(\sigma, n + b, q)} \nonumber \\
    &= \sum_{\tau \in s(p(\nu))E^{N}} s_{(p(\mu)\tau, m + b, q)} s^*_{(p(\nu)\tau, n+b, q)}.\label{eq:halfwaymark}
\end{align}
Likewise,
\begin{equation}\label{eq:halfwaymark2}
s_\alpha s^*_\beta = \sum_{\tau \in s(p(\beta))E^{N}} s_{(p(\alpha)\tau, a + n, q)} s^*_{(p(\beta)\tau, b + n, q)}.
\end{equation}
We have $p(\mu) = p(\alpha)$ and $p(\nu) = p(\beta)$ by assumption. So it suffices to show that
$a+n = m+b$. That $s(\mu) = s(\nu)$ implies in particular that $q + c(\mu) = q + c(\nu)$ and hence
that for $j \not= k$,
\[
c_0(p(\mu))_j - \sum_{i \not= j} m_i = c_0(p(\nu)) - \sum_{i \not= j} n_i.
\]
Similarly, each
\[
c_0(p(\alpha))_j - \sum_{i \not= j} a_i = c_0(p(\beta)) - \sum_{i \not= j} b_i.
\]
Since $p(\mu) = p(\alpha)$ and $p(\nu) = p(\beta)$, we may subtract the two equations above to
obtain
\[
\sum_{i \not= j} (m_i - a_i) = \sum_{i \not= j} (n_i - b_i) \quad\text{ for all $j \not= k$}.
\]
Moreover, $\sum^k_{i=1} a_i = N = \sum^k_{i=1} m_i$, and similarly for the $n_i$ and $b_i$, so we
obtain $m_j - a_j = n_j - b_j$ for all $j < k$; and then $m_k - a_k = n_k - b_k$ also because $|m|
= |n| = |a| = |b| = N$. So $m - a = n - b$, and rearranging we obtain $a+n = m+b$, proving the
claim.

For $\eta,\zeta \in p(v)E^N$, Lemma~\ref{lem:a,b exist} yields $a,b \in \NN^k$ with $|a| = |b| = N$
and $c(\eta,a) = c(\zeta,b)$. We then have $s(\eta,a,q) = s(\zeta,b,q)$ if and only if $s(\eta) =
s(\zeta)$. For each pair $\eta, \zeta \in p(v)E^n$ such that $s(\eta) = s(\zeta)$, define
$w_{\eta,\zeta} := s_{(\eta,a,q)} s^*_{(\zeta,b,q)}$. The above claim shows that the
$w_{\eta,\zeta}$ depend only on $\eta$ and $\zeta$ and not on our choice of $a$ and $b$. The claim
also implies that $s_\mu s^*_\nu = w_{(p(\mu), p(\nu))}$ whenever $\mu,\nu \in v\Lambda^N$ with
$s(\mu) = s(\nu)$ (this forces $s(p(\mu)) = s(p(\nu))$. Hence
\[
    B_N = \lsp\{w_{\eta, \zeta} : \eta,\zeta \in p(v)E^N, s(\eta) = s(\zeta)\}.
\]
The $w_{\eta,\zeta}$ are nonzero because $s_\mu s^*_\nu \not= 0$ in $C^*(\Lambda)$ whenever $s(\mu)
= s(\nu)$ \cite[Remarks~1.6(iv)]{KP2000}. It remains to check that they are matrix units. We have
\[
w^*_{\eta,\zeta} = (s_{(\eta, a ,q)} s^*_{(\zeta, b, q)})^* = s_{(\zeta, b, q)} s^*_{(\eta,a,q)} = w_{\zeta,\eta}
\]
for any choice of $a,b$ for which this makes sense. Lemma~\ref{lem:equal pis} implies that
$w_{\eta,\zeta} w_{\alpha,\beta} = 0$ if $\alpha \not= \zeta$.

Suppose that $\alpha = \zeta$. Let $a,b,m,n \in \NN^k$ be the unique elements such that $|a| = |b|
= |m| = |n| = N$ and $a_j = c(\beta)_j$, $b_j = c(\alpha)_j$, $m_j = c_0(\zeta)_j$ and $n_j =
c(\eta_j)$ for $j < k$. So by Lemma~\ref{lem:a,b exist} and Lemma~\ref{lem:equal pis}, we have
$w_{\alpha,\beta} = s_{(\alpha,a,q)}s^*_{(\beta,b,q)}$ and $w_{\eta,\zeta} =
s_{(\eta,m,q)}s^*_{(\zeta,n,q)}$. Since $\alpha = \zeta$, Lemma~\ref{lem:equal pis} and the
composition formula in $\Lambda$ gives
\begin{align}
w_{\eta,\zeta} w^*_{\alpha,\beta}
    &= \sum_{\tau \in s(\zeta) E^N} s_{(\eta\tau, m+a,q)} s^*_{(\beta\tau, n+b,q)} \nonumber\\
    &= \sum_{\tau \in s(\zeta) E^N} s_{(\eta,a,q)} s_{(\tau,m, q + c(\eta,a))}
        s^*_{(\tau, b, q + c(\beta, n))} s^*_{(\beta, n,q)} \nonumber\\
    &= s_{(\eta,a,q)} \Big(\sum_{\tau \in s(\zeta) E^N} s_{(\tau, m, q + c(\eta,a))}
        s^*_{(\tau, b, q + c(\beta, n))}\Big) s^*_{(\beta, n,q)}
        \label{eq:product of mu's}
\end{align}
We claim that $c(\eta,a) = c(\beta,n)$. We have $c(\eta,a)_k = N = c(\beta,n)_k$. Fix $j < k$. Then
\begin{align*}
c(\eta,a)_j
    &= c_0(\eta)_j - N + a_j \\
    &= c_0(\eta)_j - N + m_j + (a_j - m_j) \\
    &= c_0(\zeta)_j - N + n_j + (a_j - b_j) \quad\text{by definition of $m,n$}.
\end{align*}
The symmetric calculation gives $c(\beta,n)_j = c_0(\alpha)_j - N + a_j + (n_j - m_j)$.  Since
$\alpha = \zeta$, we have $b = m$ also, so $c(\eta,a)_j = c(\beta,n)_j$ as claimed.

We now have $s(\eta,a,q) = s(\beta,n,q)$, so that $w_{\eta,\beta} =
s_{(\eta,a,q)}s^*_{(\beta,n,q)}$, and~\eqref{eq:product of mu's} becomes
\begin{align*}
w_{\eta,\zeta} w_{\alpha,\beta}
    &= s_{(\eta,a,q)}
        \Big(\sum_{\tau \in s(\zeta) E^N} s_{(\tau, m, q + c(\eta,a))} s^*_{(\tau, m, q + c(\eta, a))}\Big)
            s^*_{(\beta, n,q)} \\
    &= s_{(\eta,a,q)} \Big(\sum_{\lambda \in s(\eta,a,q)\Lambda^m} s_\lambda s^*_\lambda\Big) s^*_{(\beta, n,q)},
\end{align*}
which is equal to $s_{(\eta,a,q)}s^*_{(\beta,n,q)}$ by~(CK4), and hence to $w_{\eta,\beta}$. This
proves that $B_N(v)$ is finite-dimensional with matrix units as claimed.

The indicated inclusion $B_N(v) \subseteq B_{N+1}(v)$ is an immediate consequence of the
Cuntz-Krieger relations. To see that $s_v C^*(\Lambda) s_v$ is the closure of the union of the
$B_N$, observe that $s_v C^*(\Lambda) s_v$ is spanned by elements of the form $s_\mu s^*_\nu$ where
$\mu,\nu \in v\Lambda$ and $s(\mu) = s(\nu)$. Fix such a spanning element. Writing $v = (p(v), 0,
q)$, we have $\mu = (p(\mu), d(\mu), q)$ and $\nu = (p(\nu), d(\nu), q)$ and then
\begin{align*}
    s(\mu) = s(\nu) \implies c(p(\mu), d(\mu)) = c(p(\nu), d(\nu)) &\implies |p(\mu)| = |p(\nu)|\\
        & \implies s_\mu s^*_\nu \in B_{|p(\mu)|}(v).
\end{align*}
So $\bigcup_{N = 1}^\infty B_N(v)$ contains all the spanning elements of $s_v C^*(\Lambda) s_v$,
whence its closure is equal to $s_v C^*(\Lambda) s_v$.

It is routine to check that, for each $N \in \NN$, the set $A_N(p(v)) := \{s_\alpha s^*_\beta :
\alpha,\beta \in p(v)E^N, s(\alpha) = s(\beta)\}$ is a set of nonzero matrix units for a
finite-dimensional subalgebra of $C^*(E)^\gamma$ and that $s_{p(v)}C^*(E)^\gamma s_{p(v)}$ is the
closure of the increasing union of the $A_N$ with inclusions $s_\alpha s^*_\beta \mapsto \sum_{f
\in s(\alpha)E^1} s_{\alpha f} s^*_{\alpha f}$. So $w_{\alpha,\beta} \mapsto s_\alpha s^*_\beta$
determines isomorphisms $B_N(v) \to A_N(p(v))$ which respect the inclusion maps. So the inductive
limits $s_v C^*(\Lambda) s_v$ and $s_{p(v)} C^*(E)^\gamma s_{p(v)}$ are isomorphic also. For the
final statement observe that the proof of \cite[Proposition~4.8]{KP2000} shows that if $\Lambda$ is
cofinal then every $s_v$ is full in $C^*(\Lambda)$, and hence $C^*(\Lambda)$ is strongly Morita
equivalent to $s_v C^*(\Lambda) s_v$ for any $v$.
\end{proof}

\section{Higher-rank graphs with finitely many vertices}\label{sec:unital}

In this section, we completely characterise the higher-rank graphs with finitely many vertices
whose $C^*$-algebras are AF. We then go on to prove that the standard dichotomy for simple graph
$C^*$-algebras persists for row-finite locally convex $k$-graphs with finitely many vertices, and
we describe the structure of non-simple unital finite higher-rank graph $C^*$-algebras.

\begin{rmk}
A standard argument \cite[Proposition~1.4]{KPR1998} implies that if $\Lambda$ is a finitely aligned
$k$-graph, then $C^*(\Lambda)$ is unital if and only if $\Lambda^0$ is finite. So one may regard
the results in this section as results about unital $k$-graph $C^*$-algebras.
\end{rmk}

In the sequel we denote by $\Kk(\Hh)$ the $C^*$-algebra of compact operators on a separable Hilbert
space $\Hh$.  When $\Hh$ has finite dimension $n$ we identify $\Kk(\Hh)$ with $M_n(\CC)$ in the
canonical way. Furthermore, given a countable set $S$, for each $\alpha,\beta \in S$,
$\theta_{\alpha,\beta}$ denotes the canonical matrix unit in $\Kk(\ell^2(S))$. The following
theorem extends~\cite[Lemma~4.2]{EvansPhD}.

\begin{theorem}\label{thm:finite k-graphs}
Let $\Lambda$ be a finitely aligned $k$-graph such that $\Lambda^0$ is finite. Then
\begin{enumerate}
\item $C^*(\Lambda)$ is AF if and only if $\Lambda$ contains no cycles, and
\item $C^*(\Lambda)$ is finite-dimensional if and only if $\Lambda$ contains no cycles and is
    row-finite, in which case there is an isomorphism
    \[
        \bigoplus_{v \in \Lambda^0, v\Lambda = \{v\}} M_{\Lambda v}(\CC) \cong C^*(\Lambda)
    \]
    which takes $\theta_{\alpha,\beta}$ to $s_\alpha s^*_\beta$.
\end{enumerate}
\end{theorem}

Before proving the theorem, we establish two technical results. Recall from \cite{Sims2006} that
\emph{satiated} collections $\Ee$ (see \cite[Definition~4.1]{Sims2006}) of finite exhaustive
subsets of $\Lambda$ index the relative Cuntz-Krieger algebras $C^*(\Lambda;\Ee)$ which interpolate
between the Toeplitz algebra $\Tt C^*(\Lambda)$ and the Cuntz-Krieger algebra $C^*(\Lambda)$
\cite[Corollary~5.6]{Sims2006}. Moreover, all quotients of $C^*(\Lambda)$ by gauge-invariant ideals
can be realised as relative Cuntz-Krieger algebras associated to complements of saturated
hereditary subgraphs \cite[Theorem~5.5]{Sims2006a}.

\begin{lemma}\label{lem:1-vert ideal AF}
Let $\Lambda$ be a finitely aligned $k$-graph. Suppose that $\Lambda^0$ is finite, and that
$\Lambda$ contains no cycles. Let $\Ee$ be a satiated subset of $\FE(\Lambda)$. If $v \in
\Lambda^0$ satisfies $v \Lambda = \{v\}$, then $C^*(\Lambda; \Ee) s_v C^*(\Lambda; \Ee) \cong
\Kk(\ell^2(\Lambda v))$.
\end{lemma}
\begin{proof}
Since $v \Lambda = \{v\}$ for $\mu,\nu \in \Lambda v$, we have
\[
\Lmin(\mu,\nu)
    = \begin{cases}
        \{(v,v)\} &\text{ if $\mu = \nu$} \\
        \emptyset &\text{ otherwise.}
    \end{cases}
\]
In particular, the relative Cuntz-Krieger relations imply that if $\mu,\nu,\alpha,\beta \in \Lambda
v$, then $s_\mu s^*_\nu s_\alpha s^*_\beta = \delta_{\nu,\alpha} s_\mu s^*_\beta$. Hence
$C^*(\Lambda; \Ee) s_v C^*(\Lambda; \Ee) = \clsp \{s_\mu s^*_\nu : \mu,\nu \in \Lambda v\}$, and
that there is an isomorphism of $\Kk(\ell^2(\Lambda v))$ with $C^*(\Lambda;\Ee) s_v C^*(\Lambda;
\Ee)$ which takes $\theta_{\mu,\nu}$ to $s_\mu s^*_\nu$.
\end{proof}

\begin{prop}\label{prp:finite vertices AF}
Let $\Lambda$ be a finitely aligned $k$-graph such that $\Lambda^0$ is finite and such that
$\Lambda$ contains no cycles, and let $\Ee$ be a satiated subset of $\FE(\Lambda)$. Then
$C^*(\Lambda; \Ee)$ is AF.
\end{prop}
\begin{proof}
We proceed by induction on $|\Lambda^0|$. If $|\Lambda^0| = 1$, then since $\Lambda$ has no cycles,
$\Lambda = \{v\}$ where $v$ is the unique element of $\Lambda^0$, and hence $C^*(\Lambda) = \CC$ is
certainly AF.

Now suppose that for any finitely aligned $k$-graph $\Gamma$ with no cycles and with fewer vertices
than $\Lambda$, and for any satiated subset $\Ee'$ of $\FE(\Gamma)$, the $C^*$-algebra $C^*(\Gamma;
\Ee')$ is AF.  Let $\{s_{\lambda} : \lambda \in \Lambda\}$ denote the universal generating relative
Cuntz-Krieger $(\Lambda;\Ee)$-family in $C^*(\Lambda; \Ee)$. Since $\Lambda^0$ is finite, and since
$\Lambda$ contains no cycles, there exists $v \in \Lambda^0$ such that $v\Lambda = \{v\}$, and then
Lemma~\ref{lem:1-vert ideal AF} implies that the ideal $I = C^*(\Lambda; \Ee) s_v C^*(\Lambda;
\Ee)$ is AF. Let $H := \{v \in \Lambda^0 : s_v \not\in I\}$ and let
\[\textstyle
\Ee' := \big\{E \in \FE(\Lambda) \setminus \Ee : \prod_{\lambda \in E}
(s_{r(\lambda)} - s_\lambda s^*_\lambda) \in I\big\}.
\]
If $H = \emptyset$, then $1_{C^*(\Lambda; \Ee)} = \sum_{v \in \Lambda^0} s_v \in I$, so
$C^*(\Lambda; \Ee) = I$ is AF and we are done. So suppose that $H \not= \emptyset$. Let $\Gamma :=
\Lambda \setminus \Lambda H$. An application of the gauge-invariant uniqueness theorem
\cite[Theorem~6.1]{Sims2006} for relative Cuntz-Krieger algebras shows that $C^*(\Gamma; \Ee')
\cong C^*(\Lambda; \Ee)/I$. Moreover $\Gamma^0 \subset \Lambda^0 \setminus \{v\}$, so $\Gamma$ has
fewer vertices than $\Lambda$. The inductive hypothesis therefore implies that $C^*(\Lambda;
\Ee)/I$ is AF. Since $I$ is AF and the class of AF algebras is closed under extensions (see, for
example, \cite[Theorem~III.6.3]{Davidson1996}), it follows that $C^*(\Lambda)$ is itself AF.
\end{proof}

\begin{proof}[Proof of Theorem~\ref{thm:finite k-graphs}]
(1) If $\Lambda$ contains a cycle, then Theorem~\ref{thm:main necessary} implies that
$C^*(\Lambda)$ is not AF; and if $\Lambda$ contains no cycle, then $C^*(\Lambda)$ is AF by
Proposition~\ref{prp:finite vertices AF} applied with $\Ee = \FE(\Lambda)$.

(2) First suppose that $\Lambda$ is not row-finite. Then there exist $v \in \Lambda^0$ and $n \in
\NN^k$ such that $v\Lambda^n$ is infinite. Hence $\{s_\lambda s^*_\lambda : \lambda \in
v\Lambda^n\}$ is an infinite family of mutually orthogonal nonzero projections in $C^*(\Lambda)$,
whence $C^*(\Lambda)$ is not finite-dimensional. Now suppose that $\Lambda$ is row-finite and
contains no cycles. Let $\Lambda^0_\sources$ denote the collection of vertices $v \in \Lambda^0$
such that $v\Lambda = \{v\}$. Since $\Lambda$ contains no cycles, $\Lambda^n = \emptyset$ whenever
$|n| \ge |\Lambda^0|$. Since $\Lambda^0$ is finite and $\Lambda$ is row-finite, $\Lambda$ itself is
finite. In particular, $\Lambda \Lambda^0_\sources$ is finite. Fix $w \in \Lambda^0$. We claim that
$w\Lambda \Lambda^0_\sources$ is exhaustive. Indeed, fix $\lambda \in w\Lambda$. As above, the set
$\{n \in \NN^k : s(\lambda) \Lambda^n \not= \emptyset\}$ is bounded; let $n$ be a maximal element
of this set, and fix $\tau \in s(\lambda)\Lambda^n$. By definition of $n$, we have $s(\tau) \in
\Lambda^0_\sources$, so $\lambda\tau \in w\Lambda \Lambda^0_\sources$ trivially has a common
extension with $\lambda$. By definition of $\Lambda^0_\sources$, as in Lemma~\ref{lem:1-vert ideal
AF} we have $s^*_\mu s_\nu = \delta_{\mu,\nu} s_{s(\mu)}$ for $\mu,\nu \in \Lambda^0_\sources$.
Hence \cite[Proposition~3.5]{RSY2004} implies that
\[
s_w = \sum_{\lambda \in w\Lambda\Lambda^0_\sources}
    s_\lambda s^*_\lambda \prod_{\lambda\lambda' \in w\Lambda\Lambda^0_\sources}
            s_{\lambda\lambda'} s^*_{\lambda\lambda'}
    = \sum_{\lambda \in w\Lambda\Lambda^0_\sources} s_\lambda s^*_\lambda.
\]
Hence
\[
C^*(\Lambda) = \bigoplus_{v \in \Lambda^0_\sources}
            \lsp\{s_\alpha s^*_\beta : \alpha,\beta \in \Lambda v\}.
\]
Lemma~\ref{lem:1-vert ideal AF} implies that each $\lsp\{s_\alpha s^*_\beta : \alpha,\beta \in
\Lambda v\} \cong M_{\Lambda v}(\CC)$.
\end{proof}

We show next that for row-finite locally convex $k$-graphs with finitely many vertices, the
standard dichotomy for simple graph $C^*$-algebras persists: if $\Lambda$ is a row-finite locally
convex $k$-graph with finitely many vertices and $C^*(\Lambda)$ is simple then $C^*(\Lambda)$ is
either finite-dimensional or purely infinite. It seems likely that a similar result holds for
arbitrary $k$-graphs with finitely many vertices (though ``finite dimensional'' would be replaced
with ``isomorphic to $\Kk(H)$ for some finite- or countably-infinite-dimensional Hilbert space''),
but the arguments provided here would require substantial modification. We first need two technical
results.

\begin{lemma} \label{lem:P_rho def}
Let $\Lambda$ be a row-finite $k$-graph, and suppose that $\rho \in \Lambda$ is a cycle with no
entrance. For each $m \in \NN^k$ such that $m \wedge d(\rho) = 0$, define a map $P_\rho :
v\Lambda^m \to v\Lambda^m$ by $P_\rho(\mu) := (\rho\mu)(0, m)$. Then $P_\rho$ is bijective.
\end{lemma}
\begin{proof}
Fix $\mu\in v\Lambda^m$. Since $\rho$ has no entrance, $\Lmin(\rho,\mu) \not= \emptyset$. Fix
$(\sigma,\tau) \in \Lmin(\rho,\mu)$. Then in particular, $\tau \in s(\mu)\Lambda^m$. Now
$(\mu\tau)(0, d(\rho)) = \rho$ because $\rho$ does not have an entrance. Hence $\mu =
P_\rho((\mu\tau)(d(\rho), d(\rho)+m))$. Since $\mu \in v\Lambda^m$ was arbitrary, it follows that
$P_\rho$ is surjective. Since $\Lambda$ is row-finite, $v\Lambda^n$ is finite, so that
$P_\rho|_{v\Lambda^n}$ is surjective implies that it is bijective.
\end{proof}

\begin{lemma} \label{lem:unique entrance}
Let $\Lambda$ be a row-finite $k$-graph, and suppose that $\rho \in \Lambda$ is a cycle with no
entrance. For each $\mu \in r(\rho)\Lambda$ such that $d(\mu) \wedge d(\rho) = 0$ and for each $n
\in \NN$, there is a unique element of $s(\mu)\Lambda^{n d(\rho)}$. Moreover, there exists $p \in
\NN$ such that the unique element of $s(\lambda) \Lambda^{p d(\rho)}$ is a cycle.
\end{lemma}
\begin{proof}
Fix $\mu \in r(\rho)\Lambda$ such that $d(\mu) \wedge d(\rho) = 0$, and let $m = d(\mu)$. Observe
that for each $n \in \NN$,
\begin{equation}\label{eq:P^n_rho}
P^n_\rho(\mu)
    = P_\rho(P^{n-1}_\rho(\mu))
    = (\rho P^{n-1}_\rho(\mu))(0, m)
    = \dots
    = (\rho^n\mu)(0,m).
\end{equation}

Fix $n \in \NN$. Let $\tau_n := \big(\rho^n P^{-n}_\rho(\mu)\big)(m, m + nd(\rho))$. Since $\mu =
P^n_\rho(P^{-n}_\rho(\mu)) = \big(\rho^n P^{-n}_\rho(\mu)\big)(0,m)$, we have $\mu\tau_n = \rho^n
P^n_\rho(\mu)$, and in particular, $\tau_n \in s(\mu)\Lambda^{n d(\rho)}$. To see that
$s(\mu)\Lambda^{n d(\rho)} = \{\tau_n\}$, let $\lambda \in s(\mu)\Lambda^{n d(\rho)}$. Then
$(\mu\lambda)(0, n d(\rho)) = \rho^n$ because $\rho$ has no entrance. Let $\alpha := (\mu\lambda)(n
d(\rho), m + n d(\rho))$, so $\rho^n\alpha = \mu\lambda$. Then $P^n_\rho(\alpha) = \mu$
by~\eqref{eq:P^n_rho}, so $\alpha = P^{-n}_\rho(\mu)$, and hence $\mu\lambda = \rho^n\alpha =
\rho^n P^{-n}_\rho(\mu)$. Thus $\lambda = \tau_n$.

Since $\Lambda^m$ is finite, there exist $l, n \in \NN$ with $l < n$ such that $P^l_{\rho}(\mu) =
P^n_{\rho(\mu)}$. Let $p := n-l$. Then
\[
\mu = P^{-n}_\rho(P^n_\rho(\mu)) = P^{-n}_\rho(P^l_\rho(\mu)) = P^{-(n-l)}_\rho(\mu) = P^p_\rho(\mu),
\]
and then by definition of the $\tau_n$, we have
\[
s(\tau_p)
    = s\big(\big(\rho^p P^{-p}_\rho(\mu)\big)(m, m + p d(\rho))\big)
    = s(P^{-(n-l)}_\rho(\mu))
    = s(\mu)
    = r(\tau_p). \qedhere
\]
\end{proof}

In the following proof and some later results, given a cycle $\tau$ in a $k$-graph $\Lambda$, we
write $\tau^\infty$ for the unique element of $W_\Lambda$ such that $d(\tau^\infty)_i$ is equal to
$\infty$ when $d(\tau)_i > 0$ and equal to $0$ when $d(\tau)_i = 0$, and such that
$(\tau^\infty)(n\cdot d(\tau), (n+1)\cdot d(\tau)) = \tau$ for all $n \in \NN$.

\begin{cor}
Let $\Lambda$ be a row-finite locally convex $k$-graph such that $|\Lambda^0|$ is finite and
$C^*(\Lambda)$ is simple. If $\Lambda$ contains no cycles, then $\Lambda^0$ contains a unique
source $v$, and $C^*(\Lambda) \cong M_{\Lambda v}(\CC)$. Otherwise, $C^*(\Lambda)$ is purely
infinite.
\end{cor}
\begin{proof}
Suppose that $\Lambda$ does not contain a cycle. Then $\Lambda$ is finite by
\cite[Remark~4.1]{EvansPhD}, and then \cite[Lemma~4.2]{EvansPhD} shows that $C^*(\Lambda)$ is equal
to the direct sum over all sources $w$ in $\Lambda^0$ of $M_{\Lambda w}(\CC)$. Since $C^*(\Lambda)$
is simple, there can be just one summand, and the result follows.

Now suppose that $\Lambda$ contains a cycle. Since $C^*(\Lambda)$ is simple, $\Lambda$ is cofinal
and has no local periodicity by \cite[Theorem~3.4]{RS2009}. Hence if $\Lambda$ contains a cycle
with an entrance, then \cite[Proposition~8.8]{Sims2006a} implies that $C^*(\Lambda)$ is purely
infinite. It therefore suffices to show that $\Lambda$ contains a cycle with an entrance.

We suppose for contradiction that no cycle in $\Lambda$ has an entrance. For each cycle $\lambda
\in \Lambda$ let $I_\lambda := \{i \le k : d(\lambda)_i \not= 0\}$, and fix a cycle $\rho$ in
$\Lambda$ such that $I_\rho$ is maximal with respect to set inclusion amongst the sets $I_\lambda$.
We claim that if $\mu \in r(\rho)\Lambda$ and $m < n \le d(\mu)$, then $\mu(m) \not= \mu(n)$. To
see this, suppose for contradiction that $\mu(m) = \mu(n)$. By Lemma~\ref{lem:unique entrance}
there exist $p \in \NN \setminus\{0\}$ and a cycle $\tau$ of degree $p d(\rho)$ with $r(\tau) =
\mu(m)$. Since $d(\mu) \wedge d(\rho) = 0$, we have $(n-m) \wedge d(\rho) = 0$, so $\tau \mu(m,n)$
is a cycle with $I_{\tau \mu(m,n)} = I_\tau \sqcup I_{\mu(m,n)} \supsetneq I_\tau = I_\rho$,
contradicting our choice of $\rho$.

Since $\Lambda^0$ is finite, it follows that there exists $\mu \in r(\rho)\Lambda$ such that
$d(\mu) \wedge d(\rho) = 0$ and such that $s(\mu)\Lambda^{e_i} = \emptyset$ whenever $e_i \wedge
d(\rho) = 0$. Another application of Lemma~\ref{lem:unique entrance} implies that there exists $p
\in \NN^k$ and a cycle $\tau \in s(\mu)\Lambda^{p d(\rho)}$. Since $r(\tau)\Lambda^{e_i} =
\emptyset$, the graph morphism $\tau^\infty$ belongs to $\Lambda^{\le \infty}$, and since cycles in
$\Lambda$ have no entrance, $s(\mu)\Lambda^{\le\infty} = \{\tau^\infty\}$. In particular,
$\sigma^{d(\tau)}(x) = x$ for all $x \in s(\mu)\Lambda^{\le\infty}$, which contradicts that
$\Lambda$ has no local periodicity.
\end{proof}

We conclude the section with the following description of the $C^*$-algebras of row-finite locally
convex $k$-graphs with finitely many vertices: each such $C^*$-algebra either contains an infinite
projection or is strongly Morita equivalent (denoted $\Me$) to a direct sum of matrix algebras over
the continuous functions on tori of dimension at most $k$ (with the convention that a dimension
zero torus is a point). To prove the result, we need some terminology. Let $\Lambda$ be a
row-finite locally convex $k$-graph such that $|\Lambda^0|$ is finite, and suppose that
$C^*(\Lambda)$ does not contain an infinite projection. We will call paths $\mu$ such that $r(\mu)
= s(\mu)$ and $r(\mu) \Lambda^{e_i} = \emptyset$ whenever $d(\mu)_i = 0$ \emph{initial cycles}, and
we will say that a vertex $v \in \Lambda^0$ is a vertex on the initial cycle $\mu$ if $v \in
(\mu^\infty)^0 := \{\mu^{\infty}(n) : n \le d(\mu^\infty)\}$. We write $\IC(\Lambda)$ for the
collection of initial cycles in $\Lambda$, and $\IC(\Lambda)^0$ for the collection of vertices of
$\Lambda$ which lie on an initial cycle.

\begin{lemma}\label{lem:G_mu}
Let $\Lambda$ be a row-finite locally convex $k$-graph such that $|\Lambda^0|$ is finite, and
suppose that $C^*(\Lambda)$ does not contain an infinite projection. Let $\mu$ be an initial cycle
of $\Lambda$. Let $G_\mu := \{m - n : m,n \le d(\mu^\infty), \mu^\infty(m) = \mu^{\infty}(n) \}$.
Then $G_\mu$ is a subgroup of $\ZZ^k$.
\end{lemma}
\begin{proof}
It is clear that $0 \in G$ and that $-G = G$, so we just need to show that $G$ is closed under
addition. Suppose that $\mu^\infty(m) = \mu^{\infty}(n)$ and that $\mu^\infty(p) = \mu^\infty(q)$,
so that $m-n$ and $p-q$ are elements of $G$; we must show that $(m-n) + (p-q) \in G$. We calculate:
\[
\mu^{\infty}(m+p)
    = \sigma^{m+p}(\mu^\infty)(0)
    = \sigma^m(\sigma^p(\mu^\infty))(0)
    = \sigma^m(\sigma^q(\mu^\infty))(0)
    = \sigma^{m+q}(\mu^\infty)(0).
\]
A symmetric argument shows that $\mu^{\infty}(n+q) = \sigma^{m+q}(\mu^\infty)(0)$ also. Hence
$(m-n) + (p-q) = (m+p) - (n+q) \in G$ as required.
\end{proof}

\begin{prop}
Let $\Lambda$ be a row-finite locally convex $k$-graph such that $|\Lambda^0|$ is finite, and
suppose that $C^*(\Lambda)$ does not contain an infinite projection. Then there exist $n \ge 1$ and
$l_1, \dots, l_n \in \{0, \dots, k\}$ such that $C^*(\Lambda) \Me \bigoplus^n_{i=1} C(\TT^{l_i})$.
\end{prop}
\begin{proof}
Since $C^*(\Lambda)$ contains no infinite projection, Lemma~\ref{lem:compare projs} implies that no
cycle in $\Lambda$ has an entrance.

For $p \in \NN$, let $\mathbf{p} := (p,p,\dots,p) \in \NN^k$. Let $N := |\Lambda^0|$. Fix $\lambda
\in \Lambda^{\le \mathbf{N}}$. Since $N = |\Lambda^0|$, there exist $p < q \le N$ such that the
vertices $\lambda(\mathbf{p}\wedge d(\lambda))$ and $\lambda(\mathbf{q} \wedge d(\lambda))$
coincide. By \cite[Lemma~3.12 and Lemma~3.6]{RSY2003}, the path $\mu := \lambda(\mathbf{p}\wedge
d(\lambda), \mathbf{q} \wedge d(\lambda))$ belongs to $\Lambda^{\le \mathbf{p - q}}$, so
$r(\mu)\Lambda^{e_i} = \emptyset$ whenever $d(\mu)_i = 0$. Since $\mu$ has no entrance,
$r(\mu)\Lambda^n = \{\mu^{\infty}(0,n)\}$ for all $n \le d(\mu^\infty)$.

By the preceding paragraph, for every $\lambda \in \Lambda^{\le \mathbf{N}}$, we have $s(\lambda)
\in \IC(\Lambda)^0$. By the Cuntz-Krieger relations,
\[
\sum_{\lambda \in \Lambda^{\le \mathbf{N}}} s_\lambda s_{s(\lambda)} s^*_\lambda
    = \sum_{v \in \Lambda^0} \sum_{\lambda \in v\Lambda^{\le \mathbf{N}}} s_\lambda s^*_\lambda
    = \sum_{v \in \Lambda^0} p_v = 1_{C^*(\Lambda)},
\]
so $\sum_{v \in \IC(\Lambda)^0} s_v$ is a full projection in $C^*(\Lambda)$. For each initial cycle
$\mu$, we write $P_\mu$ for $\sum_{v \in (\mu^\infty)^0} s_v$.

Given two initial cycles $\mu, \nu$ either $(\mu^\infty)^0 = (\nu^\infty)^0$, or $(\mu^\infty)^0
\cap (\nu^\infty)^0 = \emptyset$. We write $\mu \sim \nu$ if $(\mu^\infty)^0 = (\nu^\infty)^0$.
Since cycles in $\Lambda$ have no entrance, if $\mu, \nu \in \IC(\Lambda)$, with $\mu \not\sim
\nu$, then $v \Lambda w = \emptyset$ for all $v \in (\mu^\infty)^0$ and $w \in (\nu^\infty)^0$, and
hence $P_\mu C^*(\Lambda) P_\mu \perp P_\nu C^*(\Lambda) P_\nu$. In particular,
\begin{align*}
C^*(\Lambda)
    \Me \Big(\sum_{v \in \IC(\Lambda)^0} s_v\Big) C^*(\Lambda) \Big(\sum_{v \in \IC(\Lambda)^0} s_v\Big)
    &= \sum_{[\mu] \in \IC(\Lambda)/\sim} P_\mu C^*(\Lambda) P_\mu \\
    &= \bigoplus_{[\mu] \in \IC(\Lambda)/\sim} P_\mu C^*(\Lambda) P_\mu.
\end{align*}
It therefore suffices to show that for $\mu \in \IC(\Lambda)$, we have $P_\mu C^*(\Lambda) P_\mu
\Me C(\TT^l)$ for some $l \le k$.

For this, fix $\mu \in \IC(\Lambda)$. For $v \in (\mu^\infty)^0$, we have $v = \mu^{\infty}(m)$ for
some $m$, and then $s_v = s_{\mu(0,m)}^* s_{\mu(0,m)} =  s_{\mu(0,m)}^* s_{r(\mu)} s_{\mu(0,m)}$,
so $s_{r(\mu)}$ is full in $P_\mu C^*(\Lambda) P_\mu$. It therefore suffices to show that
$s_{r(\mu)} C^*(\Lambda) s_{r(\mu)} \cong C(\TT^l)$ for some $l \le k$. By
\cite[Corollary~3.7]{Allen2008}, $s_{r(\mu)} C^*(\Lambda) s_{r(\mu)}$ is isomorphic to the
universal $C^*$-algebra generated by elements $\{t_{\alpha,\beta} : r(\alpha) = r(\beta) = r(\mu),
s(\alpha) = s(\beta)\}$ such that
\begin{enumerate}
\item $t^*_{\alpha,\beta} = t_{\beta,\alpha}$,
\item $t_{\alpha,\beta} t_{\eta,\zeta} = \sum_{(\tau,\rho) \in \Lmin{\beta,\eta}}
    t_{\alpha\tau,\zeta\rho}$, and
\item for every finite exhaustive subset $E$ of $r(\mu)\Lambda$, $\prod_{\lambda \in E}
    (t_{v,v} - t_{\lambda,\lambda}) = 0$.
\end{enumerate}
Since $\mu$ has no entrance, relation~(3) holds if and only if each $t_{\alpha,\alpha} = t_{v,v}$,
and then~(2) implies that $t_{v,v}$ is a unit for the corner, and that each $t_{\alpha,\beta}$ is a
unitary. If $\alpha \in r(\mu)\Lambda$, then $\alpha = \mu^\infty(0,d(\alpha))$. For $m,n \le
d(\mu^\infty)$ with $\mu^\infty(m) = (\mu^\infty)(n)$, let $\alpha = \mu^\infty(0,m - m \wedge n)$
and $\beta = \mu^\infty(0,n - m \wedge n)$. Then~(2) implies that
\[
t_{\alpha,\beta} - t_{\mu^\infty(0,m),\mu^\infty(0,n)}
    = t_{\alpha,\beta} (t_{r(\mu),r(\mu)} - t_{\mu^\infty(0,n),\mu^\infty(0,n)}),
\]
and since $\{\mu^\infty(0,n)\}$ is exhaustive in $r(\mu)\Lambda$, it follows that $t_{\alpha,\beta}
= t_{\mu^\infty(0,m),\mu^\infty(0,n)}$. In particular, if $G_\mu$ is the group obtained from
Lemma~\ref{lem:G_mu}, then there is a well-defined function $(m - n) \mapsto u_{m-n} :=
t_{\mu^\infty(0,m),\mu^\infty(0,n)}$ from $G_\mu$ to $s_{r(\mu)} C^*(\Lambda) s_{r(\mu)}$.

For $\alpha,\beta, \eta, \zeta, \tau$ and $\rho$ as in~(2), we have
\begin{align*}
d(\alpha\tau) - d(\zeta\rho)
    &= \big(d(\alpha) + (d(\beta) \vee d(\zeta)) - d(\beta)\big)
        + \big(d(\zeta) + (d(\beta) \vee d(\zeta)) - d(\zeta)\big)\\
    &= (d(\alpha) - d(\beta)) + (d(\eta) - d(\zeta)),
\end{align*}
and hence for $g,h \in G_\mu$ we have $u_g u_h = u_{g+h}$.

Hence $s_{r(\mu)} C^*(\Lambda) s_{r(\mu)}$ is the universal $C^*$-algebra generated by a unitary
representation of $G_\mu$, namely $C^*(G_\mu)$. Since $G_\mu$ is a subgroup of $\ZZ^k$ it is
isomorphic to $\ZZ^l$ for some $l \le k$, so $C^*(G_\mu) \cong C(\TT^l)$ as required.
\end{proof}

\section{Examples}\label{sec:examples}

In this final section, we present some examples which illustrate our results. We begin with an
example that illustrates the need for the fairly technical definition of a generalised cycle.

Before we discuss it, recall that the $C^*$-algebra of a directed graph is AF if and only if the
graph contains no cycle. There are two obvious generalisations of the notion of a cycle to the
setting of $k$-graphs: paths whose range and source coincide, or periodic infinite paths. Examples
have appeared previously in the literature to show that there exist $k$-graphs containing no path
whose range and source coincide whose $C^*$-algebras are not AF (for example the pullback of
$\Omega_1$ by the homomorphism $(p,q) \mapsto p+q$; see \cite[Example~1.7, Definition~1.9, and
Corollary~2.5(iii)]{KP2000}) and that there exist $k$-graphs in which every infinite path is
aperiodic and the $C^*$-algebra is not AF (see \cite[Examples 6.5~and~6.6]{PRRS2006}). However, to
our knowledge, the following is the first known example of a $k$-graph which contains no
(conventional) cycle and in which every infinite path is aperiodic but such that the $C^*$-algebra
is not AF. This confirms the conjecture stated at the opening of \cite[Section~4.1]{EvansPhD}.

\begin{example}\label{non-AF aper}
Let $\Lambda$ be the $2$-graph with skeleton
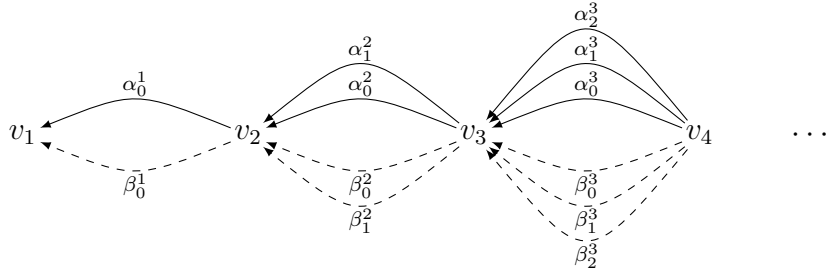
\[
\begin{tikzpicture}[scale=1.5]
    \node[circle,inner sep=.5pt] (v0) at (0,0) {$v_1$};
    \node[circle,inner sep=.5pt] (v1) at (2,0) {$v_2$};
    \node[circle,inner sep=.5pt] (v2) at (4,0) {$v_3$};
    \node[circle,inner sep=.5pt] (v3) at (6,0) {$v_4$};
    \node at (7,0) {$\cdots$};
    \draw[-latex] (v1) .. controls +(-1,0.4) .. (v0)
        node[pos=0.5, anchor=south, inner sep=1pt] {$\scriptstyle{\alpha^1_0}$};
    \draw[-latex, dashed] (v1) .. controls +(-1,-0.4) .. (v0)
        node[pos=0.5, anchor=north, inner sep=1pt] {$\scriptstyle{\beta^1_0}$};
    \draw[-latex] (v2) .. controls +(-1,0.4) .. (v1)
        node[pos=0.5, anchor=south, inner sep=1pt] {$\scriptstyle{\alpha^2_0}$};
    \draw[-latex] (v2) .. controls +(-1,0.8) .. (v1)
        node[pos=0.5, anchor=south, inner sep=1pt] {$\scriptstyle{\alpha^2_1}$};
    \draw[-latex, dashed] (v2) .. controls +(-1,-0.4) .. (v1)
        node[pos=0.5, anchor=north, inner sep=1pt] {$\scriptstyle{\beta^2_0}$};
    \draw[-latex, dashed] (v2) .. controls +(-1,-0.8) .. (v1)
        node[pos=0.5, anchor=north, inner sep=1pt] {$\scriptstyle{\beta^2_1}$};
    \draw[-latex] (v3) .. controls +(-1,0.4) .. (v2)
        node[pos=0.5, anchor=south, inner sep=1pt] {$\scriptstyle{\alpha^3_0}$};
    \draw[-latex] (v3) .. controls +(-1,0.8) .. (v2)
        node[pos=0.5, anchor=south, inner sep=1pt] {$\scriptstyle{\alpha^3_1}$};
    \draw[-latex] (v3) .. controls +(-1,1.2) .. (v2)
        node[pos=0.5, anchor=south, inner sep=1pt] {$\scriptstyle{\alpha^3_2}$};
    \draw[-latex, dashed] (v3) .. controls +(-1,-0.4) .. (v2)
        node[pos=0.5, anchor=north, inner sep=1pt] {$\scriptstyle{\beta^3_0}$};
    \draw[-latex, dashed] (v3) .. controls +(-1,-0.8) .. (v2)
        node[pos=0.5, anchor=north, inner sep=1pt] {$\scriptstyle{\beta^3_1}$};
    \draw[-latex, dashed] (v3) .. controls +(-1,-1.2) .. (v2)
        node[pos=0.5, anchor=north, inner sep=1pt] {$\scriptstyle{\beta^3_2}$};
\end{tikzpicture}
\]
and factorisation rules $\alpha^i_j \beta ^{i+1}_k \sim \beta^i_{j+1 (\mod i)} \alpha^{i+1}_{k+1
(\mod i+1)}$. Wright's argument \cite{pp_Wright2011} shows that $\Lambda$ is aperiodic in the sense
of \cite{LS2010}, meaning that every vertex receives at least one aperiodic infinite path. However,
we claim that it has the stronger property that every infinite path in $\Lambda$ is aperiodic. (For
$1$-graphs this is equivalent to requiring that the graph contains no cycles; it is also equivalent
to the condition that the associated groupoid is principal.)

To see this, fix $x \in \Lambda^{\le\infty}$, say $r(x) = v_{i-1}$, and factorise $x$ as
\[
x = \alpha^i_{j_0} \beta^{i+1}_{j_1} \alpha^{i+2}_{j_2}
\beta^{i+3}_{j_3} \dots
\]

Then
\begin{align*}
\sigma^{e_1}(x)
    &= \beta^{i+1}_{j_1} \alpha^{i+2}_{j_2} \beta^{i+3}_{j_3} \alpha^{i+4}_{j_4} \dots \\
    &= \alpha^{i+1}_{j_1 - 1 (\mod i+1)} \beta^{i+2}_{j_2 - 1 (\mod i+2)}
        \alpha^{i+3}_{j_3 - 1 (\mod i+3)} \beta^{i+4}_{j_4 - 1 (\mod i+4)} \dots
\end{align*}
and
\begin{align*}
\sigma^{e_2}(x)
    &= \sigma^{e_2}(\beta^i_{j_0 + 1 (\mod i)} \alpha^{i+1}_{j_1 + 1 (\mod i+1)}
        \beta^{i+2}_{j_2 + 1 (\mod i+2)} \alpha^{i+3}_{j_3 + 1 (\mod i+3)}
        \beta^{i+4}_{j_4 + 1 (\mod i+4)}) \dots \\
    &= \alpha^{i+1}_{j_1 + 1 (\mod i+1)} \beta^{i+2}_{j_2 + 1 (\mod i+2)}
        \alpha^{i+3}_{j_3 + 1 (\mod i+3)} \beta^{i+4}_{j_4 + 1 (\mod i+4)}\dots
\end{align*}
Hence for $m,n \in \NN$,
\begin{align*}
\sigma^{(m,n)}(x) =
    {}&\alpha^{i+m+n}_{j_{m+n} + (n - m) (\mod i+m+n)}\\
      &\qquad \beta^{i+m+n+1}_{j_{m+n+1} + (n - m)(\mod i+m+n+1)} \\
      &\qquad\qquad \alpha^{i+m+n+2}_{j_{m+n+2} + (n - m) (\mod i+m+n+2)} \\
      &\qquad\qquad\qquad \beta^{i+m+n+3}_{j_{m+n+3} + (n - m) (\mod i+m+n+3)} \dots
\end{align*}

So for $p \in \NN^2$, we can recover $p$ from $y:=\sigma^{p}(x)$ as follows:
\begin{itemize}
\item if $r(x) = v_l$ and $r(\sigma^p(x)) = v_{l'}$, then $p_1 + p_2 = l' - l$.
\item for $i \ge 0$, we have $y((i,i), (i+1,i)) = \alpha^{r(y)+2i}_{k(i)}$ for some $k(i) \in
    \ZZ/(r(y) + 2i)\ZZ$. Moreover, $p_2 - p_1 \equiv k(i) - j_{r(y) + 2i} (\mod r(y) + 2i)$ for
    all $i$. Hence the sequence $d_i := k(i) - (j_{r(y) + 2i} \in \ZZ)$ is either constant or
    else increases by $2$ at each step. We have $-(r(y) + 2i) < p_2 - p_1 < r(y) + 2i$ for $i >
    0$, so if $(d_i)$ is constant then $p_2 \ge p_1$ and $p_2 - p_1 = d_i$ for $i \ge 1$, and
    if $d_{i+1} = d_i + 2$ for all $i$, then $p_2 < p_1$ and $p_2 - p_1 = d_i - (r(y) + 2i)$
    for $i \ge 1$.
\item We now know $p_1 + p_2$ and $p_2 - p_1$; we then have $p_2 = \frac{(p_1 + p_2) + (p_2 -
    p_1)}{2}$, and then $p_1 = p_1 + p_2 - p_2$.
\end{itemize}

In particular, if $\sigma^p(x) = \sigma^q(x)$, then $p = q$ by the above, and it follows that $x$
is not periodic. Hence $\Lambda$ has no periodic boundary paths. It also has no cycles. However,
$(\alpha^1_0, \beta^1_0)$ is a generalised cycle, so $C^*(\Lambda)$ is not AF. Since $g(v_n) :=
1/(n-1)!$ defines a finite faithful graph trace on $\Lambda$, Lemma~\ref{lem:graph traces
revisited} implies that $C^*(\Lambda)$ carries a faithful trace, and hence is finite.
\end{example}

\begin{rmk}
The $2$-graph of the preceding example contains a generalised cycle, so we were able to use
Theorem~\ref{thm:main necessary} to see that its $C^*$-algebra was not AF. We believe that it is
possible to construct a similar example which contains no generalised cycle and no periodic paths
whose $C^*$-algebra is simple and finite but not AF, using Proposition~\ref{prp:arbitrary partial
unitary} in place of Theorem~\ref{thm:main necessary}.
\end{rmk}

Our second example demonstrates a $2$-graph which contains no generalised cycle, but so that a
quotient graph does contain such a cycle. In particular Corollary~\ref{cor:GCs in quotients} is a
genuinely stronger result than Theorem~\ref{thm:main necessary}.

\begin{example}\label{eg:spine}
Consider the unique $2$-graph $\mathcal{S}$ with the skeleton illustrated in
Figure~\ref{fig:spine}.
\begin{figure}[ht]
\[
\begin{tikzpicture}
    \filldraw[black!15!white] (-5.5,-1.5)--(-1,-1.5)--(-1,7.5)--(-5.5,7.5)--cycle;
    \foreach \x in {-4,-2,0,2,4} {
        \foreach \y in {0,2,4,6} {
            \node[circle, inner sep=1.5pt, fill=black] (v\x\y) at (\x,\y) {};
            \ifnum \y > 0 {
                \ifnum \x = 0 {
                    \draw[-latex] (v0\y) .. controls +(0.25,-.95) .. +(0,-1.9)
                        \ifnum \y=2 node[pos=0.5, anchor=west, inner sep=1pt]  {$\alpha$}\fi;
                    \draw[-latex, dashed] (v0\y) .. controls +(-0.25,-.95) .. +(0,-1.9)
                        \ifnum \y=2 node[pos=0.5, anchor=east, inner sep=1pt] {$\beta$}\fi;
                } \else {
                    \ifnum \x > 0 {
                        \draw[-latex, dashed] (v\x\y)--+(0,-1.9);
                    } \else {
                        \draw[-latex] (v\x\y)--+(0,-1.9);
                    }\fi
                } \fi
            } \fi
            \ifnum \x < 0 {
                \draw[-latex, dashed] (v\x\y)--+(1.9,0);
            } \fi
            \ifnum \x > 0 {
                \draw[-latex] (v\x\y)--+(-1.9,0);
            } \fi
        }
        \node at (\x,-1) {$\vdots$};
        \node at (\x,7) {$\vdots$};
    }
    \foreach \y in {0,2,4,6} {
        \node at (-5,\y) {$\dots$};
        \node at (5,\y) {$\dots$};
    }
\end{tikzpicture}
\]
\caption{The skeleton of the $2$-graph $\mathcal{S}$.}\label{fig:spine}
\end{figure}
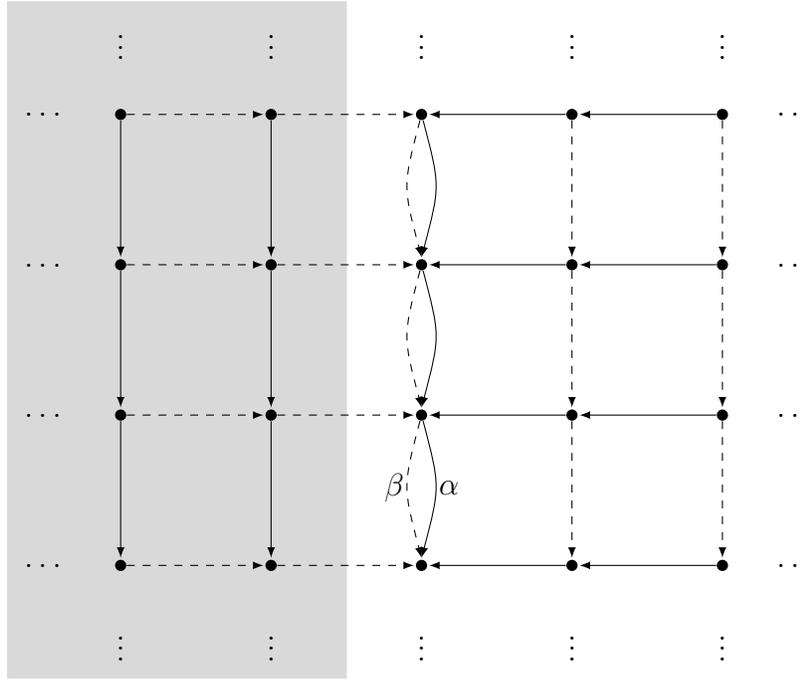
It is straightforward to check that this $2$-graph contains no generalised cycles. However, the
collection $H$ of vertices to the left of the middle (those contained in the grey rectangle) form a
saturated hereditary subset of $\Lambda^0$, and the quotient graph does contain a generalised
cycle, namely $(\alpha,\beta)$.
\end{example}

We now present two examples of $2$-graphs with the same skeleton, one of them AF, the other not
obviously so. The AF example is a $2$-graph which satisfies Condition~$(\Gamma)$ of
\cite[Definition~4.6]{EvansPhD} but not Condition~(S) \cite[Definition~4.3]{EvansPhD}, confirming a
conjecture of the first author. The other example is intriguing, because it strongly suggests that
there are $2$-graph $C^*$-algebras $C^*(\Lambda)$ which are AF but whose canonical diagonal
subalgebras (in the sense of Kumjian) as AF algebras are not conjugate to their maximal abelian
subalgebras $\clsp\{s_\lambda s^*_\lambda : \lambda \in \Lambda\}$. (By contrast, whenever the
$C^*$-algebra of a directed graph $E$ is AF, it has an AF decomposition for which the canonical
diagonal subalgebra is precisely $\clsp\{s_\lambda s^*_\lambda : \lambda \in E^*\}$, where $E^*$ is
the finite path space of $E$.)

\begin{example}\label{eg:PAF}
Consider the skeleton
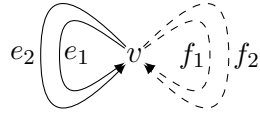
\[
    \begin{tikzpicture}[rotate=90]
    \node[inner sep = 0.5pt] (v) at (0,0) {$v$};
    \draw[-latex] (v) .. controls (0.5,0.7) and (0.7,1) .. (0,1)
        node[pos=1,anchor=west,inner sep=1.5pt] {\small$e_1$}
        .. controls (-0.7,1) and (-0.5,0.7) .. (v);
    \draw[-latex] (v) .. controls (0.8,0.8) and (1,1.25) .. (0,1.25)
        node[pos=1,anchor=east,inner sep=1.5pt] {\small$e_2$}
        .. controls (-1,1.25) and (-0.8,0.8) .. (v);
    \draw[-latex, dashed] (v) .. controls (0.5,-0.7) and (0.7,-1) .. (0,-1)
        node[pos=1,anchor=east,inner sep=1.5pt] {\small$f_1$}
        .. controls (-0.7,-1) and (-0.5,-0.7) .. (v);
    \draw[-latex, dashed] (v) .. controls (0.8,-0.8) and (1,-1.25) .. (0,-1.25)
        node[pos=1,anchor=west,inner sep=1.5pt] {\small$f_2$}
        .. controls (-1,-1.25) and (-0.8,-0.8) .. (v);
    \end{tikzpicture}
\]

Let $\PAF$ be the $2$-graph with this skeleton and factorisation rules $e_i f_j = f_i e_j$. By
\cite[Corollary~3.5(iii)]{KP2000}, $C^*(\PAF) \cong \Oo_2 \otimes C(\TT)$.

Let $c(e_1) = c(f_2) := (0,1)$, $c(e_2) := (1,1)$ and $c(f_1) := (-1,1)$. It is straightforward to
check that $c$ extends to a functor on $\PAF$. The skew-product graph $\LAF := \PAF \rtimes_c
\ZZ^2$ has the skeleton illustrated in Figure~\ref{fig:LAF/LNAF}.
\begin{figure}[ht]
\[
\begin{tikzpicture}[scale=1.75]
    \foreach \x in {-3,-2,-1,0,1,2,3} {
        \foreach \y in {-1,0,1,2} {
            \node[inner sep = 1pt] (\x\y) at (\x,\y) {\small$v_{(\x,\y)}$};
        }
        \node at (\x,2.4) {$\vdots$};
        \node at (\x,-1.25) {$\vdots$};
    }
    \foreach \y in {-1,0,1,2} {
        \node at (-3.5,\y) {\dots};
        \node at (3.5,\y) {\dots};
    }
    \foreach \x/\xx in {-3/-2,-2/-1,-1/0,0/1,1/2,2/3} {
        \foreach \y/\yy in {0/-1,1/0,2/1} {
            \draw[-latex, dashed] (\x\y)--(\xx\yy);
        }
    }
    \foreach \x in {-3,-2,-1,0,1,2,3} {
        \foreach \y/\yy in {0/-1,1/0,2/1} {
        \draw[-latex] (\x\y) .. controls +(0.15,-0.5) .. (\x\yy);
        \draw[-latex, dashed] (\x\y) .. controls +(-0.15,-0.5) .. (\x\yy);
        }
    }
    \foreach \x/\xx in {-2/-3,-1/-2,0/-1,1/0,2/1,3/2} {
        \foreach \y/\yy in {0/-1,1/0,2/1} {
            \draw[-latex] (\x\y)--(\xx\yy);
                node[pos=0.5, anchor=north west, inner sep=1pt, label=e\xx\yy2mid] {};
        }
    }
\end{tikzpicture}
\]
\caption{The common skeleton of the $2$-graphs $\LAF$ and $\LNAF$.}\label{fig:LAF/LNAF}
\end{figure}
To keep notation compact, we write $e^{i,j}_l$ for $(e_l, (i,j))$ and $f^{i,j}_l$ for $(f_l,
(i,j))$ for all $l \in \{0,1\}$ and $i,j \in \ZZ$. So locally, the labelling looks like
\[
\begin{tikzpicture}[scale=1.75]
    \node[inner sep = 1pt] (o) at (0,0) {\small$v_{(i,j)}$};
    \node[inner sep = 1pt] (nw) at (-1,1) {\small$v_{(i-1,j+1)}$};
    \node[inner sep = 1pt] (n) at (0,1) {\small$v_{(i,j+1)}$};
    \node[inner sep = 1pt] (ne) at (1,1) {\small$v_{(i+1,j+1)}$};
    \draw[-latex, dashed] (nw)--(o)
        node[pos=0.5, anchor=north east, inner sep=1pt] {\small$f^{i,j}_1$};
    \draw[-latex] (n) .. controls +(0.15,-0.5) .. (o)
        node[pos=0.3, anchor=west, inner sep=1pt] {\small$e^{i,j}_1$};
    \draw[-latex, dashed] (n) .. controls +(-0.15,-0.5) .. (o)
        node[pos=0.3, anchor=east, inner sep=1pt] {\small$f^{i,j}_2$};
    \draw[-latex] (ne)--(o)
        node[pos=0.5, anchor=north west, inner sep=1pt] {\small$e^{i,j}_2$};
\end{tikzpicture}
\]
The factorisation rules are
\begin{align*}
e^{i,j}_1f^{i,j+1}_2 &= f^{i,j}_1e^{i-1,j+1}_2, &&& e^{i,j}_2f^{i+1,j+1}_1 &= f^{i,j}_2e^{i,j+1}_1,\\
e^{i,j}_1 f^{i, j+1}_1 &= f^{i,j}_1 e^{i-1,j+1}_1, &&& e^{i,j}_2f^{i+1,j+1}_2 &= f^{i,j}_2 e^{i, j+1}_2.
\end{align*}
We claim that $C^*(\LAF)$ is strongly Morita equivalent to the UHF algebra of type $2^\infty$. To
see this, we invoke Corollary~\ref{cor:skew-pullback-structure}. Let $B_2$ be the $1$-graph with
$B_2^0 = \{v\}$ and $B_2^1 = \{a,b\}$ whose $C^*$-algebra is canonically isomorphic to $\Oo_2$.
Then $e_1 \to (a,(1,0))$, $e_2 \to (b, (1,0))$, $f_1 \to (a,(0,1))$ and $f_2 \to (b, (0,1))$
determines an isomorphism of $\PAF$ with the pullback $f^*(B_2)$ under the homomorphism $f : (m,n)
\to m+n$ from $\NN^2$ to $\NN$. The $2$-graph $\LAF$ is isomorphic to the one obtained from
Example~\ref{eg:skew-pullback} by setting $c_0(a) = e_1$ and $c_0(b) = 0$ in $\NN$. Since $\LAF$ is
cofinal, it follows from Corollary~\ref{cor:skew-pullback-structure} that $C^*(\LAF)$ is strongly
Morita equivalent to $s_v C^*(B_2)^\gamma s_v$. The fixed-point algebra $C^*(B_2)^\gamma$ is
precisely the classical core of $\Oo_2$, which is the $2^\infty$ UHF algebra \cite[1.5]{Cuntz1977}.
\end{example}

\smallskip

\begin{example}\label{eg:PNAF}
Consider the $2$-graph $\PNAF$ with the same skeleton as $\PAF$ but with factorisation rules $e_i
f_j = f_j e_i$. This is isomorphic to $B_2 \times B_2$, so $C^*(\PNAF) \cong \Oo_2 \otimes \Oo_2$
as in \cite[Corollary~3.5(iv)]{KP2000}. The formula given for $c$ in Example~\ref{eg:PAF} also
extends to a functor on $\PNAF$, and we write $\LNAF$ for the corresponding skew-product graph.
Then $\LNAF$ has the same skeleton as $\LAF$, but factorisation rules
\begin{align*}
e^{i,j}_1f^{i,j+1}_2 &= f^{i,j}_2 e^{i,j+1}_1, &&& e^{i,j}_2f^{(i+1),(j+1)}_1 &= f^{i,j}_1e^{i-1,j+1}_2,\\
e^{i,j}_1 f^{i, j+1}_1 &= f^{i,j}_1 e^{i-1,j+1}_1, &&& e^{i,j}_2f^{i+1,j+1}_2 &= f^{i,j}_2 e^{i, j+1}_2.
\end{align*}
For each $i,j$, let $x^{i,j}$ denote the unique infinite path $x^{i,j} : \Omega_2 \to \LNAF$ such
that
\[
x^{i,j}(n, n+(1,0)) = e^{i,(j+|n|)}_1
    \qquad\text{ and }\qquad
x^{i,j}(n, n+(0,1)) = f^{i,(j+|n|)}_2
\]
for all $n \in \NN^2$. Then $\sigma^{(1,0)}(x^{i,j}) = x^{i,(j+1)} = \sigma^{(0,1)}(x^{i,j})$ for
all $i,j$, and in particular every vertex of $\LNAF$ receives a periodic infinite path. On the
other hand, for each $i,j$ there is a unique infinite path $y^{i,j} : \Omega_2 \to \LNAF$ defined
by
\[
y^{i,j}(n, n+(1,0)) := e^{(i + n_1 - n_2), (j+|n|)}_2
    \qquad\text{ and }\qquad
y^{i,j}(n, n+(0,1)) := f^{(i + n_1 - n_2), (j+|n|)}_1
\]
for all $n \in \NN^2$. Since each $y^{i,j}$ is injective from $\Omega_2^0 \to \LNAF^0$ it is
aperiodic. So $\LNAF$ satisfies the aperiodicity condition. It is also cofinal, so $C^*(\LNAF)$ is
simple by \cite[Proposition~4.8]{KP2000}.
\end{example}

\subsection{The \texorpdfstring{$C^*$}{C*}-algebra of \texorpdfstring{$\LNAF$}{Lambda II}}
We will spend some time analysing $C^*(\LNAF)$. We believe that it is isomorphic to $C^*(\LAF)$,
but via an isomorphism which cannot easily be described in terms of the presentation of each as a
$k$-graph $C^*$-algebra. As supporting evidence for this conjecture, setting $p := s_{v_{(0,0)}}$,
we prove: that $p C^*(\LNAF) p$ has a unique tracial state; that $C^*(\LNAF)$ (and thus $p
C^*(\LNAF) p$) is AF-embeddable; that the $K$-theory of both $p C^*(\LNAF) p$ and $C^*(\LNAF)$ is
$(\ZZ[\frac{1}{2}], \{0\})$ (as groups); that Murray-von Neumann equivalence in $p C^*(\LNAF)p$ of
the canonical representatives of the generators of its $K_0$-group is equivalent to $K_0$
equivalence characterised by equality under the trace;  and that the order on its $K_0$-group is
the standard unperforated order. So all the evidence suggests that $C^*(\LNAF)$ is strongly Morita
equivalent to the $2^\infty$ UHF algebra, and hence also to $C^*(\LAF)$. To indicate why this might
be surprising, we close by showing that if $p C^*(\LNAF) p$ is indeed the $2^\infty$ UHF algebra,
then its diagonal subalgebra as an AF algebra is not conjugate to the canonical maximal abelian
subalgebra $\clsp\{s_\lambda s^*_\lambda : \lambda \in v_{(0,0)} \LNAF\}$, even though the two
subalgebras are canonically isomorphic under an isomorphism which preserves $K_0$-classes in the
enveloping algebras.

Recall that a \emph{normalised trace} on $C^*(\LNAF)$ is a trace such that $\sum_{v \in F}
\tau(s_v)$ converges to $1$ as $F$ increases over finite subsets of $\LNAF^0$ and that for a
hereditary subset $H$ of $\LNAF^0$, we may identify $C^*(H\LNAF)$ with the subalgebra of
$C^*(\LNAF)$ generated by $\{ s_\lambda : \lambda \in H\LNAF \}$.

\begin{lemma}\label{lem:PNAF unique trace}
Let $H := \{v_{(i,j)} : j \ge 0, |i| \le j\}$ be the hereditary subset of $\LNAF^0$ generated by
$v$. Let $T := \sum^\infty_{j=1} (2j-1)2^{1-j}$. There is a normalised trace $\tau$ on
$C^*(H\LNAF)$ given by $\tau(s_{v_{(i,j)}}) = \frac{1}{T}2^{i-j}$, and $\tau(s_\mu s^*_\nu) =
\delta_{\mu,\nu} \tau(s_{s(\mu)})$. Moreover, this is the unique normalised trace on $C^*(H\LNAF)$.
\end{lemma}
\begin{proof}
The function $g : v_{(i,j)} \to \frac{1}{T}2^{1-j}$ determines a normalised finite faithful graph
trace on each of $H\LNAF$ and $H\LAF$. Lemma~\ref{lem:graph traces revisited} implies that there
are faithful normalised traces $\tau^{\mathrm{I\hskip-0.5pt I}}_g : C^*(H\LNAF) \to \CC$ and
$\tau^\mathrm{I}_g : C^*(H\LAF) \to \CC$ satisfying $\tau_g(s_\mu s^*_\nu) = \delta_{\mu,\nu}
g(s(\mu))$ for all $\mu,\nu$. Since $C^*(H\LAF)$ is strongly Morita equivalent to $M_{2^\infty}$,
$\tau^\mathrm{I}_g$ is the unique such trace on $C^*(H\LAF)$, and hence $g$ is the unique
normalised finite graph trace on $H\LAF$. It is then also the unique normalised finite graph trace
on $H\LNAF$, so another application of Lemma~\ref{lem:graph traces revisited} implies that any
trace on $C^*(H\LNAF)$, which is nonzero on each $s_v$ and zero on each $s_\mu s_\nu^*$ such that
$d(\mu) \neq d(\nu)$, must agree with $\tau^{\mathrm{I\hskip-0.5pt I}}_g$.

We claim that $\tau^{\mathrm{I\hskip-0.5pt I}}_g$ is the unique trace on $C^*(H\LNAF)$. To see
this, fix a trace $\tau$ on $C^*(H\LNAF)$. By the above, it suffices to show that $\tau(s_v) \not=
0$ for all $v \in H$ and that $\tau(s_\mu s^*_\nu) = 0$ whenever $d(\mu) \not= d(\nu)$. To see that
$\tau(s_v) \not= 0$ for all $v$, fix $v \in H$. Since $\tau$ is normalised, we have $\tau(s_w)
\not= 0$ for some $w$. Since $\Lambda$ is cofinal, \cite[Remark~A.3]{LS2010} implies that there
exists $n \in \NN^2$ such that $v\Lambda s(\alpha) \not= \emptyset$ for all $\alpha \in
w\Lambda^n$. Since $s_w = \sum_{\alpha \in w\Lambda^n} s_\alpha s^*_\alpha$, there exists $\alpha
\in w\Lambda^n$ such that $\tau(s_\alpha s^*_\alpha) \not= 0$. Fix $\xi \in v\Lambda s(\alpha)$.
Then
\[
\tau(s_v) \ge \tau(s_\xi s^*_\xi) = \tau(s_\xi s^*_\alpha s_\alpha s^*_\xi)
    = \tau(s_\alpha s^*_\xi s_\xi s^*_\alpha) = \tau(s_\alpha s^*_\alpha) \not= 0.
\]

It remains to show that $\tau(s_\mu s^*_\nu) = 0$ when $d(\mu) \not= d(\nu)$. If $s(\mu) \not=
s(\nu)$, this is trivial, and if $r(\mu) \not= r(\nu)$, then the trace property gives $\tau(s_\mu
s^*_\nu) = \tau(s_\nu^* s_\mu) = \tau(s_\nu^* s_{r(\nu)} s_{r(\mu)} s_\mu) = 0$. So we may suppose
that $s(\mu) = s(\nu)$ and $r(\mu) = r(\nu)$. Factorise $\mu = \eta\alpha_0$ and $\nu =
\zeta\beta_0$ where $d(\eta) = d(\zeta) = d(\mu) \wedge d(\nu)$. Then $d(\alpha_0) \wedge
d(\beta_0) = 0$ and
\[
    \tau(s_\mu s^*_\nu)
        = \tau(s_\eta s_{\alpha_0} s^*_{\beta_0} s^*_\zeta)
        = \tau(s^*_\zeta s_\eta s_{\alpha_0} s^*_{\beta_0})
        = \delta_{\eta,\zeta} \tau(s_{\alpha_0} s^*_{\beta_0}).
\]
In particular, it suffices to show that $\tau(s_{\alpha_0} s^*_{\beta_0}) = 0$. If $r(\alpha_0)
\neq r(\beta_0)$ then by the above argument, we are done.  If not then let $K := |\alpha_0|$. Since
$r(\alpha_0) = r(\beta_0)$ and $s(\alpha_0) = s(\beta_0)$, we have $\alpha_0 = x^{i,j}(0, Ke_h)$
and $\beta_0 = x^{i,j}(0, Ke_l)$ for some $i,j \in \ZZ$ and $h,l$ such that $\{h,l\} = \{1,2\}$. By
the Cuntz-Krieger relations and the trace property,
\[
\tau(s_{\alpha_0} s^*_{\beta_0}) = \tau(s^*_{\beta_0} s_{\alpha_0})
    = \sum_{\alpha_0\alpha' = \beta_0\beta' \in \Lambda^{(K,K)}} \tau(s_{\alpha'} s^*_{\beta'}).
\]
Let $\alpha_1 = x^{i, (j+K)}(0, d(\beta_0))$ and $\beta_1 = x^{i, (j+K)}(0, d(\alpha_0))$, then
$\MCE(\alpha_0,\beta_0)=\{\alpha_0\alpha_1\} = \{\beta_0\beta_1\}$ so that
\[
\tau(s_{\alpha_0} s^*_{\beta_0}) = \tau(s_{\alpha_1} s^*_{\beta_1}).
\]
Repeating this, we obtain pairs $\alpha_n, \beta_n$ such that $r(\alpha_n) = r(\beta_n) =
s(\alpha_{n-1}) = s(\beta_{n-1})$ and $d(\alpha_n) = d(\beta_{n-1})$ and vice versa for all $n$,
and such that $\tau(s_{\alpha_n} s^*_{\beta_n}) = \tau(s_{\alpha_m} s^*_{\beta_m})$ for all $m,n$.
Now suppose that $K \not= 0$ so that $\alpha_0 \not= \beta_0$ and let $z := \tau(s_{\alpha_0}
s^*_{\beta_0})$; we must show that $z = 0$. Let $v_n = r(\alpha_n) = r(\beta_n)$ for all $n$. Since
$K \not= 0$, we have $v_m \not= v_n$ for distinct $m,n$. It follows that
\[
    \clsp\{s_{\alpha_n} s^*_{\beta_n}\} =
        \bigoplus_{n=0}^\infty s_{v_n}\big(\clsp\{s_{\alpha_n} s^*_{\beta_n}\}\big)s_{v_n}
        \subseteq \bigoplus^\infty_{n=1} s_{v_n} C^*(\LNAF) s_{v_n}.
\]
Since the $C^*$-norm on a direct sum is the supremum norm, it follows that the series
\[
\sum^\infty_{n=0} \frac{1}{n} s_{\alpha_n} s^*_{\beta_n}
\]
converges to some $S \in C^*(H\LNAF)$. By continuity of $\tau$, we have
\[
\tau(S) = \sum^\infty_{n=0} \tau(\frac{1}{n}s_{\alpha_n} s^*_{\beta_n}) = \sum^\infty_{n=0} \frac{z}{n},
\]
and this forces $z = 0$ since the harmonic series does not converge.
\end{proof}

\begin{cor}\label{cor:unique trace}
Let $\tau$ be the trace on $C^*(\LNAF)$ constructed in Lemma~\ref{lem:PNAF unique trace}. There is
a unique tracial state $\tau_0$ on $pC^*(\LNAF)p$ given by $\tau_0=\frac{1}{\tau(p)}\tau(a)$.  In
particular $pC^*(\LNAF)p$ and $C^*(\LNAF)$ are stably finite.
\end{cor}

Recall that $C^*(\LNAF)$ is the skew-product of $B_2 \times B_2$ by the $\ZZ^2$-valued functor $c$
satisfying $c(a,v) = c(v,b) = (0,1)$, $c(b,v) = (1,1)$ and $c(v,a) = (-1,1)$. We write $l$ for the
length functor $l(\alpha,\beta) := |\alpha| + |\beta|$ from $B_2 \times B_2$ to $\ZZ$, and we write
$\gamma^l$ for the corresponding induced action satisfying $\gamma^l_z(s_\lambda) =
z^{l(\lambda)}s_\lambda$.

\begin{lemma}\label{lem:fixed-point algebras v2}
The $C^*$-algebra $C^*(B_2 \times B_2)^{\gamma^l}$ is isomorphic to $\bigotimes_\ZZ M_2
\rtimes_\sigma \ZZ$, where $\sigma$ is the (Bernoulli) shift automorphism that translates each
tensor factor one position to the right.
\end{lemma}

\begin{proof}
Let $S_1$ and $S_2$ be the canonical generators of the Cuntz algebra $\Oo_2$, and let $\Ff_2$ be
the AF core of $\Oo_2$.  We will prove that $C^*(B_2 \times B_2)^{\gamma^l}$ is isomorphic to
$C^*(\Ff_2 \otimes \Ff_2 \cup \{U\}) \subset \Oo_2$, where $U:=S_1^* \otimes S_1 + S_2^* \otimes
S_2$ is unitary.  The Lemma will follow since $C^*(\Ff_2 \otimes \Ff_2 \cup \{U\}) \cong
\bigotimes_\ZZ M_2 \rtimes_\sigma \ZZ$ by \cite[Proposition~3.3]{CRS}.

First, note that $u:=s_{(v,a)}s_{(a,v)}^* + s_{(v,b)}s_{(b,v)}^*$ is a unitary in $C^*(B_2 \times
B_2)$.  We will show that
\[
C^*(B_2 \times B_2)^{\gamma^l} = C^*(C^*(\Lambda)^\gamma \cup \{u\}).
\]
For this, observe that
\[
    C^*(B_2 \times B_2)^{\gamma^l}
        = \clsp\{s_\alpha s^*_\beta : \alpha, \beta \in B_2 \times B_2, l(\alpha) = l(\beta) \}.
\]

Since $d(\alpha) = d(\beta) \implies l(\alpha) = l(\beta)$, and since $\gamma_z^l(u)=u$ for all
$z\in \TT^2$, we have $C^*(C^*(\Lambda)^\gamma \cup \{u\}) \subseteq C^*(B_2 \times
B_2)^{\gamma^l}$.

For the reverse inclusion, since $u$ is unitary, it suffices to show that for each spanning element
$s_\alpha s^*_\beta$ of $C^*(B_2 \times B_2)^{\gamma^l}$, there exists $n \in \ZZ$ such that $u^n
s_\alpha s^*_\beta \in C^*(B_2 \times B_2)^\gamma$. For this, fix $\alpha=(\alpha_1, \alpha_2),
\beta=(\beta_1, \beta_2) \in B_2 \times B_2$ such that $l(\alpha)=l(\beta)$.  For each $z \in
\TT^2$ we have $\gamma_z(u^n s_\alpha s_\beta^*) = ((z^{(-1,1)})^n u)(z^{(n,-n)} s_\alpha
s_\beta^*) = u^n s_\alpha s_\beta^* \in C^*(B_2 \times B_2)^\gamma$, as required, where
$n=|\alpha_1|-|\beta_1|=|\beta_2|-|\alpha_2|$.

By \cite[Corollary~3.5(iv)]{KP2000}, there is an isomorphism $\psi : C^*(B_2 \times B_2) \cong
\Oo_2 \otimes \Oo_2$ satisfying $\psi(s_{(\alpha,\beta)}) = S_\alpha \otimes S_\beta$ (with the
obvious identification of paths and multi-indices). Hence the restriction of $\psi$ to $C^*(B_2
\times B_2)^{\gamma^l}$ is the required isomorphism, since $\psi(u) = U$ and $\psi(C^*(B_2 \times
B_2)^\gamma)=\Ff_2 \otimes \Ff_2$.
\end{proof}

\begin{prop}\label{prp:LNAF embedding}
The $C^*$-algebra $C^*(\LNAF)$ is AF-embeddable.
\end{prop}
\begin{proof}
Since every automorphism of a UHF algebra is approximately inner, $\bigotimes_\ZZ M_2
\rtimes_\sigma \ZZ$ is AF-embeddable by \cite[3.6 Theorem]{Voi86}.  Thus Lemma~\ref{lem:fixed-point
algebras v2} implies that $C^*(B_2 \times B_2)^{\gamma^l}$ is AF-embeddable.  For all $\lambda, \mu
\in \LNAF, c(\lambda)=c(\mu) \implies l(\lambda)=l(\mu)$, from which it follows that $C^*(B_2
\times B_2)^{\gamma^c} \subseteq C^*(B_2 \times B_2)^{\gamma^l}$; thus $C^*(B_2 \times
B_2)^{\gamma^c}$ is also AF-embeddable.  By \cite[Proposition]{Rosenberg79}, $C^*(\LNAF)$ is
strongly Morita equivalent, and thus stably isomorphic, to $C^*(B_2\times B_2)^{\gamma^c}$.  Hence,
$C^*(\LNAF)$ is itself AF-embeddable.
\end{proof}

\begin{rmk}
It is known that $\bigotimes_\ZZ M_2 \rtimes_\sigma \ZZ$ is a simple, unital, (non AF)
A$\TT$-algebra of real rank zero and has a unique tracial state \cite{BratteliStormerEtAl:ETDS93}.
Thus the same is true for $C^*(B_2\times B_2)^{\gamma^l}$, which is strongly Morita equivalent to
$C^*((B_2 \times B_2) \times_l \ZZ)$ by \cite[Proposition]{Rosenberg79}.  Hence $C^*((B_2 \times
B_2) \times_l \ZZ)$ is another example of a simple, $2$-graph $C^*$-algebra that is neither AF nor
purely infinite.
\end{rmk}

To prove our K-theory results we will need the fact that the $K_0$-group of $C^*(\LNAF)$ is generated by the classes of its vertex projections.  The following Lemma proves this fact holds in general.

\begin{lemma}\label{lem:K_0}
Let $\Lambda$ be a row-finite $2$-graph with no sources. Suppose that the degree of each cycle in
$\Lambda$ has zero first coordinate or the degree of each cycle has zero second coordinate.  Then
$K_0(C^*(\Lambda))$ is generated by $\{[s_v] : v\in \Lambda^0\}$.
\end{lemma}

\begin{proof}
As in \cite[Definition~3.6]{Evans2008}, for $i = 1, 2$ we let $M_i$ denote the $\Lambda^0 \times
\Lambda^0$ integer matrix $M_i(v,w) = |v\Lambda^{e_i}w|$. By our hypothesis, and without loss of
generality, the coordinate graph of $\Lambda$ corresponding to $e_1$ (see \cite[Examples
1.10.(i)]{KP2000}), which we denote by $\Lambda_1$, contains no cycles and so it is AF by
\cite[Examples 1.7.(i)]{KP2000} and \cite[Theorem 2.4]{KPR1998}.  It is well known that the
$K_1$-group of an AF algebra is trivial so that $K_1(\Lambda_1)$ is the trivial group. From
\cite[Remarks 3.19]{Evans2008} (or the well-known formulae for the K-groups of directed graph
$C^*$-algebras) we get $ \ker (1-M_1^t) \cong K_1(\Lambda_1) \cong \{0\}$.  Therefore the
block-column matrix
\[
\begin{pmatrix}
 M^t_2-1 \\ 1 - M^t_1
\end{pmatrix}
\]
of \cite[Proposition~3.16]{Evans2008} has trivial kernel, and the Lemma now follows from
\cite[Proposition~3.16]{Evans2008}.
\end{proof}

\begin{prop}\label{prp:LNAF K-theory}
The $K_0$-group $K_0(p C^*(\LNAF) p)$ is generated by the classes $\{[s_\lambda s^*_\lambda] :
\lambda \in v_{(0,0)}\Lambda\}$, and two such classes are equal if and only if the associated
projections are Murray-von Neumann equivalent in $p C^*(\LNAF) p$. Let $\tau_0$ be the unique
tracial state on $p C^*(\LNAF) p$ defined in Corollary \ref{cor:unique trace}. Then $\tau_0$
induces an order-isomorphism $K_0(\tau_0)$ between the ordered $K_0$-group of $pC^*(\LNAF)p$ and
$(\ZZ[\frac{1}{2}], \ZZ[\frac{1}{2}]\cap \RR^+)$, and hence an order-isomorphism between the
ordered $K_0$-group of $C^*(\LNAF)$ and $(\ZZ[\frac{1}{2}], \ZZ[\frac{1}{2}]\cap \RR^+)$. In
particular $K_0(\tau_0)$ carries the class of the identity to $1$, and carries $\{[s_\lambda
s^*_\lambda] : \lambda \in v_{0,0}\LNAF\}$ to $[0,1] \cap \ZZ[\frac12]$.  Moreover,
$K_1(C^*(\LNAF))$, and hence also $K_1(pC^*(\LNAF)p)$, is trivial.
\end{prop}
\begin{proof}
Let $A_0:=pC^*(\LNAF)p$ and $A:=C^*(\LNAF)$.  Then $A_0$ is unital and $A$ is stably unital
\cite[Remarks 1.6.(v)]{KP2000}. Moreover, both are stably finite by Corollary \ref{cor:unique
trace} so their $K_0$-groups are equipped with the canonical partial ordering \cite[\S
6.3]{Blackadar:CUP98}.  Furthermore, the trace $\tau_0$ induces a state $K_0(\tau_0)$ on
$K_0(A_0)$, whose range is $\ZZ[\frac{1}{2}]$.  Since $A_0$ is a full corner in $A$ the inclusion
mapping $i:A_0\hookrightarrow A$ induces an order-isomorphism $K_0(i): K_0(A_0) \rightarrow
K_0(A)$.

The partial isometry $(s_{e_1^{i,j}} + s_{e_2^{i,j}})(s_{f_1^{i+1,j}} + s_{f_2^{i+1,j}})^*$
implements a Murray von-Neumann equivalence in $A$ between $s_{v_{(i,j)}}$ and $s_{v_{(i+1,j)}}$
for all $i,j\in\ZZ$.  It follows, by induction, that $s_{v_{(i,j)}} \sim_{\mathrm{M.vN}}
s_{v_{(k,j)}}$ for all $i,j,k\in\ZZ$, and hence their $K_0$-classes are equal. Moreover, for each
$i,j,k\in\ZZ$ such that $k\ge j$ we have
\[
[s_{v_{(i,j)}}]
    =\sum_{\lambda \in v_{(i,j)}\LNAF^{(k-j)e_1}} [s_\lambda s_\lambda^*]
    = |v_{(i,j)}\LNAF^{(k-j)e_1}| [ s_{v_{(i,k)}} ] = 2^{k-j} [s_{v_{(i,k)}}],
\]
since, for each $\lambda\in v_{(i,j)}\LNAF^{(k-j)e_1}$, $s(\lambda)=v_{(l,k)}$ for some $l\in\ZZ$,
and $s_\lambda s_\lambda^*$ is Murray von-Neumann equivalent to $s(\lambda)$ in $A$.

Let $i,j\in\ZZ$.  If $j<0$, then $[s_{v_{(i,j)}}] = 2^{0-j}[s_{v_{(i,0)}}] =
2^{-j}[s_{v_{(0,0)}}]$. If $j\ge 0$ then $[s_{v_{(i,j)}}]=[s_{v(0,j)}]=[s_\lambda s_\lambda^*]$ for
some $\lambda\in v_{(0,0)}\LNAF^{je_1}$.
By Lemma \ref{lem:K_0} $K_0(A)$ is generated by $\{[s_v] : v\in \LNAF^0\}$. Therefore it is also generated by
$\{[s_{\lambda} s_{\lambda}^*] : \lambda \in s_{v_{(0,0)}}\LNAF \}$.  Thus $K_0(A_0)$ is generated
by the pre-image under $K_0(i)$, namely $\{ [s_{\lambda} s_{\lambda}^*] : \lambda \in
s_{v_{(0,0)}}\LNAF \}$.

Let $\lambda,\mu\in v_{(0,0)}\LNAF$ such that $s(\lambda)=v_{(i,j)}, s(\mu)=v_{(k,l)}$ for some
$i,j,k,l\in\ZZ$ with $j\le k$.  Then the following equations are satisfied in $K_0(A)$: $[s_\lambda
s_\lambda^*]=[s_{v_{(i,j)}}]=2^{k-j}[s_{v_{(i,k)}}] = 2^{k-j} [s_\mu s_\mu^*]$. It follows that
$[s_\lambda s_\lambda^*]=2^{k-j}[s_\mu s_\mu^*]$ in $K_0(A_0)$ also.

The above implies that we can write each $x \in K_0(A_0)$ as $x=\sum_{j=0}^n x_j [s_{\lambda_j}
s_{\lambda_j}^*]$ where $\lambda_j\in v_{(0,0)}\LNAF v_{(0,j)}$ for all $j=1,2,\ldots, n$.
Furthermore, $x=(\sum_{j=0}^n 2^{n-j}x_j)[s_{\lambda_n} s_{\lambda_n}^*]$ so that each element in
$K_0(A_0)$ can be written as an integer multiple of $[s_\lambda s_\lambda^*]$ for some $\lambda\in
v_{(0,0)}\LNAF$.

We will show that $K_0(\tau_0)$ is injective. Suppose that $K_0(\tau_0)(x)=0$ for some $x\in
K_0(A_0)$. Now $x=m[s_\lambda s_\lambda^*]$ for some $m\in\ZZ$ and $\lambda\in v_{(0,0)}\LNAF$.  So
we have $0=m K_0(\tau_0)([s_\lambda s_\lambda^*])=m \tau_0(s_\lambda s_\lambda^*)$. Thus $m=0$
since $\tau_0$ is faithful, so that $x=0$, as required. Since we already showed that projections
with the same trace are Murray-von Neumann equivalent, it follows that $K_0$-equivalence is the
same as Murray-von Neumann equivalence.

We have established that $K_0(\tau_0)$ is a positive isomorphism, so to show that it is an
isomorphism of ordered groups, it suffices to prove that $\ZZ[\frac{1}{2}]^+ \subseteq
K_0(\tau)(K_0(A_0)^+)$. Let $y\in\ZZ[\frac{1}{2}]^+$ then $y=m2^{-n}$ for some $m,n\in\NN$.  But
$m2^{-n}=K_0(\tau)(m[s_\lambda s_\lambda^*])$ for some $\lambda\in v_{(0,0)}\LNAF$, thus
$y=K_0(\tau)(x)$ for some $x\in K_0(A_0)^+$, as required.

As $A_0$ is strongly Morita equivalent to $A$, it remains to prove that $K_1(A)$ is isomorphic to
the trivial group. This follows immediately from \cite[Proposition~3.16]{Evans2008} since $\LNAF$
and $\LAF$ share the same skeleton.
\end{proof}

\begin{rmk}
It follows from the above the unique normalised traces on $C^*(H\LNAF)$ and on $C^*(H\LAF)$
determine an order-isomorphism $K_0(C^*(\LNAF)) \cong K_0(C^*(\LAF))$ which carries the class of
each vertex projection in $C^*(\LNAF)$ to the class of the corresponding vertex projection in
$C^*(\LAF)$.
\end{rmk}

All of the above is strong evidence suggesting that $C^*(\LNAF)$ is isomorphic to $C^*(\LAF)$.
However, since each vertex $v_{i,j}$ receives both a periodic infinite path $x^{i,j}$ and an
aperiodic infinite path $y^{i,j}$, every open set in the unit space $\Gg^{(0)}_\Lambda$ of the
$k$-graph groupoid of \cite{KP2000} contains both a point with trivial isotropy and a point with
nontrivial isotropy. So $\Gg_\Lambda$ is topologically free but not principal: the set of units
with trivial isotropy is dense in $G^{(0)}$, but not the whole of $G^{(0)}$. Let $D :=
\clsp\{s_\lambda s^*_\lambda : \lambda \in \LNAF\}$. It follows from
\cite[Proposition~5.11]{Renault2008} that the pair $(C^*(\LNAF), D)$ is a Cartan pair but not a
$C^*$-diagonal, and in particular that $D$ does not have unique extension of pure states to
$C^*(\LNAF)$. Since the canonical maximal abelian subalgebra in an AF algebra is always a diagonal
in the sense of Kumjian and in particular always has the extension property, it follows that if
$C^*(\LNAF)$ is AF, then its canonical diagonal subalgebra (as an AF algebra) is not conjugate to
the Cartan subalgebra $D$. It must, however, be in some sense locally conjugate: the range of the
trace on $C^*(\LNAF)$ is exhausted on $D$, and since trace equivalence coincides with Murray-von
Neumann equivalence for projections in AF algebras, it would follow that each projection in the AF
diagonal was Murray-von Neumann equivalent to a projection in $D$.

We have been unable to determine whether $C^*(\LNAF)$ is indeed an AF algebra, and leave this
interesting question open.

\begin{rmk}
Our results in this paper constitute only a start on the problem of when a $k$-graph algebra is AF.
There are numerous examples upon which our results shed little light, and it seems likely that
substantially different techniques are required to understand them. One such is the unique
$2$-graph $\mathcal{X}$ with skeleton as illustrated in Figure~\ref{fig:Xgraph}
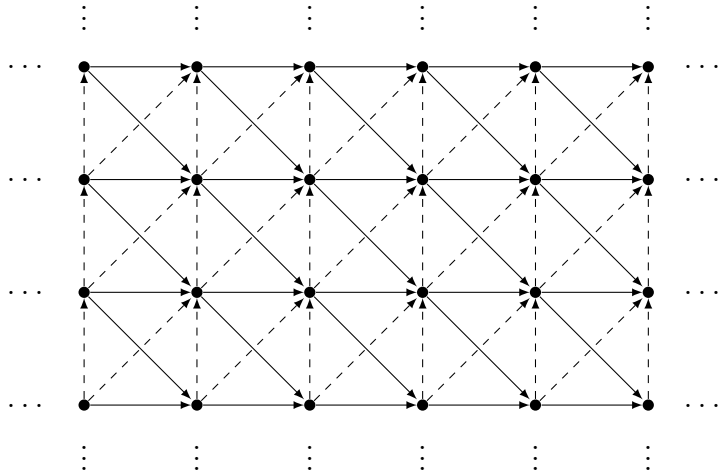
\begin{figure}[ht]
\[
\begin{tikzpicture}[scale=1.5]
    \foreach \x in {-1,0,1,2,3,4} {
        \foreach \y in {-1,0,1,2} {
            \node[inner sep=1.5pt, circle, fill=black] (\x\y) at (\x,\y) {};
            \ifnum \x < 4 {
                \draw[-latex] (\x\y)--+(0.95,0);
                \ifnum \y > -1 {
                    \draw[-latex] (\x\y)--+(0.95,-0.95);
                } \fi
                \ifnum \y < 2 {
                    \draw[-latex,dashed] (\x\y)--+(0.95,0.95);
                } \fi
            } \fi
            \ifnum \y < 2 {
                \draw[-latex,dashed] (\x\y)--+(0,0.95);
            } \fi
        }
    }
    \foreach \x in {-1,0,1,2,3,4} {
        \node at (\x,-1.4) {$\vdots$};
        \node at (\x,2.5) {$\vdots$};
    }
    \foreach \y in {-1,0,1,2} {
        \node at (-1.5,\y) {$\dots$};
        \node at (4.5,\y) {$\dots$};
    }
\end{tikzpicture}
\]
\caption{The skeleton of the $2$-graph $\mathcal{X}$.}\label{fig:Xgraph}
\end{figure}
discussed on \cite[pages 52~and~53]{EvansPhD}. We pass up, for now, analysis of this example: the
questions raised by the example $\LNAF$ above seem to go more directly to the heart of the
difficulty of the question of when a $k$-graph $C^*$-algebra is AF.
\end{rmk}

\end{document}